\documentclass[reqno]{amsart}

\usepackage[utf8]{inputenc}
\usepackage[T1]{fontenc}
\usepackage{enumerate}
\usepackage{lmodern, microtype, amssymb}

\usepackage{graphicx}
\usepackage{color,psfrag,caption}
\definecolor{darkgreen}{rgb}{0,0.6,0}
\definecolor{darkred}{rgb}{0.7,0,0}
\definecolor{darkblue}{rgb}{0,.2,.7}

 \usepackage[%
   colorlinks, 
  citecolor=darkgreen, linkcolor=darkblue,
   pdfpagelabels=true,
   unicode=true,
   pdfusetitle
 ]{hyperref}
 
 \hypersetup{
 pdfauthor={Nadine Grosse and Cornelia Schneider},
 pdftitle={Symmetries on manifolds: Generalizations of the Radial Lemma of Strauss},
 pdflang={en}
}

\renewcommand{\epsilon}{\varepsilon}
\renewcommand{\i}{\mathrm{i}}

\let\theta\vartheta
\let\phi\varphi
\let\<\langle
\let\>\rangle
\DeclareMathAlphabet{\doba}{U}{msb}{m}{n}
\gdef\mC{\doba{C}}

\gdef\mN{\doba{N}}
\gdef\mS{\doba{S}}
\gdef\mR{\doba{R}}

\def\supp{{\mathop\mathrm{supp}}}
\def\vol{{\mathop\mathrm{vol}}}
\def\vo{{\mathop\mathrm{dvol}}}
\def\d{{\text d}}

\def\a{{\mathfrak{a}}}

\newcommand{\define}{\mathrel{\mathrm{:=}}}

\newtheorem{lem}{Lemma}

\newtheorem{thm}[lem]{Theorem}

\newtheorem{assumption}[lem]{Assumption}

\theoremstyle{definition}
\newtheorem{defi}[lem]{Definition}
\newtheorem{ex}[lem]{Example}
\newtheorem{rem}[lem]{Remark}

\newcommand{\uD}{\mathrm{D}}
\newcommand{\uA}{\mathcal{A}}
\newcommand{\uT}{\mathcal{T}}
\newcommand{\beq}{\begin{equation}}
\newcommand{\eeq}{\end{equation}}
\newcommand{\real}{\mathbb{R}}

\newcommand{\rn}{\mathbb{R}^n}

\newcommand{\nat}{\mathbb{N}}

\newcommand{\F}{F^s_{p,q}(\mathbb{R}^n)}

\newcommand\mtop[2]{\genfrac{}{}{0pt}{}{#1}{#2}}

\begin{document}

\title[Symmetries on manifolds: Generalizations of the Strauss Lemma]{Symmetries on manifolds: Generalizations of the Radial Lemma of Strauss}
\author{Nadine Gro\ss{}e and Cornelia Schneider} 

\begin{abstract}
For a compact subgroup $G$ of the group of isometries acting on a Riemannian manifold $M$ we investigate subspaces of Besov and Triebel-Lizorkin type which are invariant with respect to the  group action. Our main aim is to extend the classical Strauss lemma under suitable assumptions on the Riemannian manifold by proving that $G$-invariance of functions implies certain  decay properties and better local smoothness. As an application we obtain inequalities of Caffarelli-Kohn-Nirenberg type for $G$-invariant functions. Our results generalize those obtained in \cite{Skr02}. The main tool in our investigations are  atomic decompositions adapted to the $G$-action in combination with trace theorems. 
\end{abstract}

\maketitle

\section{Introduction}

The aim of this paper is to generalize the radial lemma of Strauss. The phenomenon that the presence of symmetries improves Sobolev embeddings and leads to a surprising interplay between regularity and decay properties of radial  functions has first been observed by \textsc{Strauss} in \cite{Str77} and later on studied by several other authors, cf. \textsc{Coleman, Glaser} and \textsc{Martin} \cite{CGM78} or \textsc{Lions} \cite{Lio82}. Similar results for Sobolev spaces with symmetries on Riemannian manifolds were proved by \textsc{Hebey} and \textsc{Vaugon}, cf. \cite{Heb99}, \cite{Heb96}, and \cite{HeVa97}. Results of this kind were applied to nonlinear PDE in \cite{Aub98}, \cite{KuPo97}.   \\
 For radial subspaces  of the Besov and Triebel-Lizorkin spaces, $B^s_{p,q}(\rn)$ and $F^s_{p,q}(\rn)$, respectively,  these kind of problems were first tackled in \cite{SiSk00} and later on extended and generalized in \cite{SSV12} and  \cite{Skr02}, where the latter paper considers closed subgroups $H\subset O(n)$ and corresponding invariant functions on $\rn$. \\
We wish to continue this line of thought here and consider a Riemannian manifold $M$ of bounded geometry (see Definition~\ref{defi_bdgeom} and below) together with a compact subgroup $G$ of the group of isometries acting on $M$ and some assumption on the Riemannian metric (to be more precise, we  assume that the metric is adapted to the $G$-action, see Appendix~\ref{app} and Example~\ref{ex:rn}).  
Questions concerning continuity and compactness of embeddings  for subspaces of Besov and Triebel-Lizorkin type consisting of $G$-invariant functions were already studied in \cite{Skr03}.  Now we investigate the decay properties of $G$-invariant functions $f$ belonging to Besov and Triebel-Lizorkin  spaces  and obtain in Theorem~\ref{thm-strauss} as an extension of the Strauss lemma, that under certain restrictions on the parameters $p$ and $q$ the relation between smoothness and decay  behaves like  
\begin{equation}\label{1}
|f(x)|\leq C \vol_{n-\tilde{n}} (G\cdot x)^{-\frac 1p}\|f|B^s_{p,q}(M)\|, \qquad s>\frac{\tilde{n}}{p}, 
\end{equation}
where $\tilde{n}$ denotes the dimension of the orbit space $M/G$. Similar for the $F$-spaces. Our main tool is an optimal decomposition via heat atoms adapted to the action of the group $G$ taken from \cite{Skr03}, cf.  Theorem~\ref{thm:optimal-dec}.  Furthermore, we show in Theorem~\ref{hoelder-reg} that we have some improved local H\"older smoothness of the $G$-invariant functions belonging to $B^s_{p,q}(M)$ and $F^s_{p,q}(M)$, respectively, whenever $\supp \ f\subset {\tilde{M}}$, where $\tilde{M}$ denotes the principal stratum, cf. Section~\ref{sec:prep}. In particular, we show that in this case 
\[
f\in \mathcal{C}^{s-\frac{\tilde{n}}{p}} (\tilde{M}). 
\]
We give several examples to illustrate our results. 
As an application of \eqref{1} we obtain some Caffarelli-Kohn-Nirenberg inequalities in Theorem~\ref{thm:CKN-ineq},  stating that for $1\leq p<r<q\leq \infty$, $1\leq \tilde{q}\leq \infty$, and certain restrictions on the parameter $s$, we have 
\[
\|\vol_{n-\tilde{n}} (G\cdot x)^{\frac 1p\left(1-\frac rq\right)}f|L_q(M)\|\leq c \|f|B^s_{p,\tilde{q}}(M)\|, \qquad s-\frac np=-\frac nr.
\]
Similar for $F$-spaces. 

The paper is organized as follows. In Section~\ref{sec:2} we recall the basic definitions concerning manifolds with bounded geometry and  the function spaces we are interested in and state related atomic decompositions of the latter. Afterwards, in Section~\ref{sec:prep}, we provide an optimal atomic decomposition adapted to the $G$-action. 
For our further considerations  we will need geometric estimates on the constants coming from the $G$-adapted covering. Those will be stated blankly  in Subsection~\ref{ass} and we give first examples of manifolds that fulfill our requirements. A larger class of examples will be presented in Appendix~\ref{app}. 
The optimal atomic decomposition together with the geometric estimates then will be applied in Section~\ref{sec:4},  
where we present and prove our main results concerning decay properties and local smoothness of $G$-invariant functions. Finally, in Section~\ref{sec:5}, as an application, we obtain some Caffarelli-Kohn-Nirenberg inequalities.

\section{Preliminaries and notations}\label{sec:2}

\subsection{General notations}  Let $\nat$ be the collection of all natural numbers, and let $\nat_0 = \nat \cup \{0 \}$. Let $\rn$ 
be the $n$-dimensional Euclidean space, $n \in \nat$, $\mathbb{C}$ the complex plane, and let $B(x,r)$  denote the open ball in $\mR^n$, or later in the given manifold, with center $x$ and radius $r$. Moreover, index sets are always assumed to be countable, and we use the Einstein sum convention.

For a real number $a$, let $a_+\define \max(a,0)$, and let $[a]$ denote its
integer part. For $p\in (0,\infty]$, the number $p'$ is defined by
$1/p'\define (1-1/p)_+$ with the convention that $1/\infty=0$. All unimportant positive constants will be denoted by $c$, occasionally with
subscripts.  For two non-negative expressions ({{i.e.},
functions or functionals) ${\mathcal  A}$, ${\mathcal  B}$, the
symbol ${\mathcal A}\lesssim {\mathcal  B}$ (or ${\mathcal A}\gtrsim
{\mathcal  B}$) means that $ {\mathcal A}\leq c\, {\mathcal  B}$ (or
$c\,{\mathcal A}\geq {\mathcal B}$) for a suitable constant $c$. If ${\mathcal  A}\lesssim
{\mathcal  B}$ and ${\mathcal A}\gtrsim{\mathcal  B}$, we write
${\mathcal  A}\sim {\mathcal B}$ and say that ${\mathcal  A}$ and
${\mathcal  B}$ are equivalent. Given two (quasi-) Banach spaces $X$ and $Y$, we write $X\hookrightarrow Y$
if $X\subset Y$ and the natural embedding of $X$ into $Y$ is continuous.\\

\subsection{Manifolds of bounded geometry}

Let $M^n$ be an $n$-dimensional complete manifold with Riemannian metric $g$. 
We denote the volume element on $M$ with respect to the metric $g$ by $\vo_g$. 

A cover $(U_\alpha)_{\alpha \in I}$ of $M$ is a collection of open subsets of $U_\alpha\subset M$ where $\alpha$ runs over an index set $I$. 
The cover is called \emph{uniformly locally finite} if there exists a constant $L>0$ such that each $U_\alpha$ is intersected by at most $L$ sets $U_\beta$. 

A \emph{chart} on $U_{\alpha}$ is given by local coordinates, i.e., a diffeomorphism  $\kappa_\alpha\colon x=(x^1, \ldots, x^n) \in V_\alpha \subset \mathbb{R}^n \to \kappa_\alpha(x)\in U_\alpha$. 
A collection $\uA=(U_\alpha, \kappa_\alpha)_{\alpha\in I}$ is called an \emph{atlas of $M$}. 

Moreover, a collection of smooth functions $(h_\alpha)_{\alpha\in I}$ on $M$ with 
\[ \supp\ h_{\alpha}\subset U_{\alpha},\qquad  0\leq h_\alpha \leq 1,\qquad  \text{and} \qquad  \sum_\alpha h_\alpha=1 \quad \text{on }M  
 \]
 is called a \emph{partition of unity subordinated to the cover $(U_\alpha)_{\alpha\in I}$}.
 The triple $\uT\define (U_{\alpha}, \kappa_{\alpha},h_{\alpha})_{\alpha \in I}$ is called a \emph{trivialization} of the manifold $M$. 

\begin{defi}\label{defi_bdgeom}\cite[Def. 1.1]{Sh92}
A Riemannian manifold $(M^n,g)$ is of \emph{bounded geometry} if the following two conditions are satisfied:
\begin{itemize}
\item[(i)] The injectivity radius $r_{M}$  of $(M,g)$ is positive.
\item[(ii)] Every covariant derivative of the Riemann curvature tensor $R^M$ of $M$ is bounded, i.e., for all $k\in \mN_0$ there is a constant $C_k>0$ such that $|(\nabla^M)^kR^M|_g\leq C_k$.
\end{itemize}
\end{defi}

 \begin{ex}[{\textbf{Geodesic trivialization}}]\label{geod_coord_triv}
Let $(M,g)$ be a complete Riemannian manifold of bounded geometry. Fix $z\in M$. For $v\in T_z^{\leq r}M\define\{ w\in T_zM\ |\ g_z(w,w)\leq r^2\}$, we denote by $c_v\colon [-1,1] \to M$ the unique geodesic with $c_v(0)=z$ and $\dot{c}_v(0)=v$. Let $r>0$ be smaller than the injectivity radius of $M$. Then by definition the exponential map $\exp^M_z\colon T_z^{\leq r}M \to M$ is a diffeomorphism defined by
$\exp^M_z(v)\define c_{v}(1)$. Let $S=\{p_\alpha\}_{\alpha\in I}$ be a set of points in $M$ such that $(U^{\mathrm{geo}}_\alpha\define B_r(p_\alpha))_{\alpha\in I}$ covers $M$.  For each $p_\alpha$ we choose an orthonormal frame of $T_{p_\alpha}M$ and call the resulting identification $\lambda_\alpha\colon \mR^n\to T_{p_\alpha} M$. Then, $\uA^{\mathrm{geo}}= (U^{\mathrm{geo}}_\alpha, \kappa^{\mathrm{geo}}_\alpha=\exp_{p_\alpha}^M\circ \lambda_\alpha\colon V_\alpha^{\mathrm{geo}}\define B_r^n \to U_\alpha^{\mathrm{geo}})_{\alpha \in I}$ is an atlas of $M$ -- called geodesic atlas. { (Note that $\lambda_\alpha^{-1}$ equals the tangent map ${(d \kappa_\alpha^{\mathrm{geo}})}^{-1}$ at $p_\alpha$.)}\\
Furthermore, there exists a  geodesic atlas that is uniformly locally finite. 
Moreover, there is a partition of unity $h^{\mathrm{geo}}_\alpha$ subordinated to $(U^{\mathrm{geo}}_\alpha)_{\alpha\in I}$ such that for all $k\in \mN_0$ there is a constant $C_k>0$ such that $|\uD^\a (h^{\mathrm{geo}}_\alpha\circ \kappa^{\mathrm{geo}}_\alpha)|\leq C_k$ for all multi-indices $\a$ with $|\a|\leq k$, cf. \cite[Proposition~7.2.1]{Tri92} and the references therein.
The resulting trivialization is denoted by $\uT^{\mathrm{geo}}=(U^{\mathrm{geo}}_\alpha, \kappa^{\mathrm{geo}}_\alpha, h^{\mathrm{geo}}_\alpha)_{\alpha \in I}$ and referred to as \emph{geodesic trivialization}.
\end{ex}

In general, function spaces defined on $M$ defined via localization and pull-back onto $\rn$ do depend on the underlying trivialization $\uT$ used in the definition. The standard approach is to use the geodesic trivialization.  In order to gain some flexibility and work with Fermi coordinates we  investigated in \cite{GS13} under which conditions on $\uT$ the resulting  norms of our function spaces turn out to be equivalent. We introduced the following terminology. 

\begin{defi}\label{bddcoord} Let $(M^n,g)$ be a Riemannian manifold of bounded geometry. Moreover, let a uniformly locally finite trivialization $\uT=(U_\alpha, \kappa_\alpha,h_\alpha)_{\alpha\in I}$ be given. We say that $\uT$ is \emph{admissible} if the following conditions are fulfilled:
\begin{itemize}
 \item[(B1)] $\uA=(U_\alpha, \kappa_\alpha)_{\alpha\in I}$ is compatible with geodesic coordinates, i.e., for $\uA^{\mathrm{geo}}=(U^{\mathrm{geo}}_\beta, \kappa^{\mathrm{geo}}_\beta)_{\beta \in J}$ being a geodesic atlas of $M$ as in Example \ref{geod_coord_triv}, 
  there are constants $C_k>0$ for $k\in \mN_0$ such that for all $\alpha\in I$ and $\beta\in J$ with $U_{\alpha}\cap U^{\mathrm{geo}}_{\beta}\neq \varnothing$ and all $\a\in \mN_0^n$ with $|\a|\leq k$,  
  \[\quad \qquad  |\uD^\a (\mu_{\alpha\beta}\define (\kappa_\alpha)^{-1} \circ \kappa^{\mathrm{geo}}_\beta)|\leq C_k \text{\ \ and\ \ }  |\uD^\a (\mu_{\beta\alpha}\define (\kappa^{\mathrm{geo}}_\beta)^{-1} \circ \kappa_\alpha)|\leq C_k.\]
\item[(B2)]  For all $k\in \mathbb{N}$ there exist $c_k>0$ such that for all $\alpha\in I$ and all multi-indices $\a$ with $|\a|\leq  k$,  
\[ |D^{\a}(h_\alpha\circ\kappa_\alpha)|\leq c_k.\]
\end{itemize}
\end{defi}

\subsection{Function spaces  on manifolds with bounded geometry}

For $0<p<\infty$ the $L_p$-norm of a compactly supported smooth function $v\in \mathcal{D}(M)\define C_c^\infty(M)$ is given by $\Vert v\Vert_{L_p(M)}\define\left( \int_M |v|^p\vo_g\right)^\frac{1}{p}$. The set $L_p(M)$ is then the completion of $\mathcal{D}(M)$ with respect to the $L_p$-norm. The space of distributions on $M$ is denoted by $\mathcal{D}'(M)$.

Before we introduce Triebel-Lizorkin and Besov spaces on manifolds, we briefly recall their definition on $\rn$. 
By the Fourier-analytical approach, Triebel-Lizorkin spaces  $\F$, $s\in\real$, $0<p<\infty$,   $\ 0<q\leq\infty$, consist of 
all tempered distributions $f\in \mathcal{S}'(\rn)$ such that
\begin{equation}\label{F-rn}
\big\Vert f\big\Vert_{\F}=
\Big\Vert \Big(\sum_{j=0}^{\infty}\big|2^{js}(\varphi_j\widehat{f})^\vee(\cdot)\big|^q\Big)^{1/q} 
\Big\Vert_{L_p(\rn)}
\end{equation}
$($usual modification if $q=\infty)$ is finite. Here $\mathcal{S}'(\rn)$ denotes the space of all tempered distributions, that is the dual of the Schwartz space. Moreover,  $\ \{\varphi_j\}_{j=0}^\infty$
denotes a   \textit{smooth dyadic resolution of unity}, where  $\ \varphi_0=\varphi \in \mathcal{S}(\rn)\ $ with 
$\ 
\supp\ \varphi\subset\left\{y\in\rn : \ |y|<2\right\}\quad \mbox{and}\quad
\varphi(x)=1\quad\mbox{if}\quad |x|\leq 1$,
and for each $\ j\in\nat\;$  put $\ \varphi_j(x)=
\varphi(2^{-j}x)-\varphi(2^{-j+1}x)$.   
In general, Besov spaces on $\rn$ are defined in the same way by interchanging the order in which the $\ell_q$- and $L_p$-norms are taken in \eqref{F-rn}. Hence,  the Besov space 
$B^s_{p,q}(\rn)$, $s\in\real$, $0<p,q \leq \infty$ consists of 
all distributions $f\in \mathcal{S}'(\rn)$ such that
\begin{align*}
\big\Vert f\big\Vert_{B^s_{p,q}(\rn)}=
 \Big(\sum_{j=0}^{\infty}2^{jsq}\big\Vert(\varphi_j\widehat{f})^\vee\big\Vert_{L_p(\rn)}^q\Big)^{1/q} 
\end{align*}
$($usual modification if $p=\infty$ and/or $q=\infty)$ is finite. The scales coincide  if $p=q$, i.e., 
\[B^s_{p,p}(\rn)=F^s_{p,p}(\rn), \qquad 0<p<\infty,\]
and we extend this to $p=\infty$ by putting $F^s_{\infty,\infty}(\rn)\define B^s_{\infty,\infty}(\rn)$.  Moreover, in \cite[(2.6.5/1)]{Tri92} it is shown that 
\begin{equation}\label{zygmund-rn}
F^s_{\infty,\infty}(\rn)=B^s_{\infty,\infty}(\rn)=\mathcal{C}^s(\rn), \quad s>0, 
\end{equation}
where $\mathcal{C}^s(\rn)$, $s>0$, denote the H\"older-Zygmund spaces as defined in \cite[Sect.~1.2.2]{Tri92}.  
The scales $F^s_{p,q}(\rn)$ and  $B^s_{p,q}(\rn)$ were studied in detail in  \cite{Tri83,Tri92}, where the reader may also find further references to the literature. 

On $\mR^n$ one usually gives priority to Besov spaces, and they are mostly considered to be the simpler ones compared to Triebel-Lizorkin spaces. However,  the situation is different on manifolds, since Besov spaces lack the so-called \textit{ localization principle}, cf. \cite[Theorem~2.4.7(i)]{Tri92}. Thus, Besov spaces on a manifold $M$ are introduced via real interpolation of Triebel-Lizorkin spaces. We refer to \cite[Sect. 2.4.1]{Tri83} and \cite{Tri78} for the definition of these interpolation spaces $(.,.)_{\Theta,q}$ and details on the subject. 

\begin{defi}\label{F-koord}
Let $(M,g)$ be a Riemannian manifold of bounded geometry with an admissible trivialization $\uT=(U_\alpha, \kappa_{\alpha}, h_{\alpha})_{\alpha \in I}$, cf. Definition~\ref{bddcoord}, and  let $s\in \real$.

\begin{itemize}
\item[(i)] Let either  $0<p<\infty$, $0<q\leq \infty$ or $p=q=\infty$.  Then the space $F^{s}_{p,q}(M)$ contains all distributions $f\in \mathcal{D}'(M)$ such that 
\begin{align*}
\Vert f\Vert_{F^{s}_{p,q}(M)}\define \left(\sum_{\alpha\in I} \Vert (h_\alpha f)\circ \kappa_\alpha\Vert^p_{F^s_{p,q}(\mathbb{R}^n)}\right)^{\frac{1}{p}}
\end{align*}
is finite (with the usual modification if $p=\infty$). 
\item[(ii)] Let $0<p,q\leq \infty$,   $-\infty< s_0<s< s_1<\infty$, and $0<\Theta<1$. Then 
\[
B^{s}_{p,q}(M)\define\left(F^{s_0}_{p,p}(M), F^{s_1}_{p,p}(M)\right)_{\Theta,q}
\]
with $s=(1-\Theta)s_0+\Theta s_1$. 
\end{itemize}
\end{defi}

\begin{rem} According to their notation, Besov and Triebel-Lizorkin spaces will sometimes abbreviated by $B$- and $F$- spaces, respectively.  
Restricting ourselves to  geodesic trivializations $\uT^{\text{geo}}$, the spaces from Definition~\ref{F-koord} coincide with the spaces $F^s_{p,q}(M)$ and $B^s_{p,q}(M)$, introduced in \cite[Definition~7.2.2, 7.3.1]{Tri92}. In \cite[Thms. 14, 57]{GS13} it was shown that the spaces defined above are independent of the chosen trivializations $\uT$ as long as they are admissible. Furthermore, the space $B^{s}_{p,q}(M)$ is  independent of the chosen numbers $s_0, s_1\in \real$. 
For simplicity reasons we shall often write $A^s_{p,q}(M)$, where $A=B$ or $A=F$, refering to both spaces. \\
The function spaces on manifolds have a lot of properties analogous to the spaces $F^s_{p,q}(\rn)$ and $B^{s}_{p,q}(\rn)$, cf. \cite{Tri92}. In particular, the $F$-scale covers Lebesgue spaces $L_p(M)=F^0_{p,2}(M)$,  Sobolev spaces $W^k_p(M)=F^k_{p,2}(M)$, $1<p<\infty$, $k\in\nat$, defined in terms of covariant derivatives,  as well as Bessel potential spaces $H^s_p(M)=F^s_{p,2}(M)$ defined via the Laplace-Beltrami operator $\Delta$ of $M$. Moreover, in \cite[Thm.~7.3.1(ii)]{Tri92} it is shown that $F^s_{\infty,\infty}(M)=B^s_{\infty,\infty}(M)$ for $s\in \real$ and in good agreement  with \eqref{zygmund-rn} we define H\"older-Zygmund spaces on $M$  via
\begin{equation}\label{def-zygmund-M}
\mathcal{C}^s(M)\define F^s_{\infty,\infty}(M), \qquad s>0.
\end{equation}
In \cite[Thm.~7.5.3(ii)]{Tri92} it is shown that for $s>4$ the  spaces  $\mathcal{C}^s(M)$ also admit a characterization in terms of differences of functions. 
\end{rem}

In this paper we mainly use atomic decompositions of the spaces $A^s_{p,q}(M)$. We recall the construction from  \cite{Skr98}. Let $r<r_{\text{inj}}/3$ and 
\begin{equation*}
\Omega_j=\left\{B(x_{j,i},2^{-j}r)\right\}_{i=0}^{\infty}, \qquad j=0,\ldots, \infty,
\end{equation*}
be a sequence of uniformly locally finite coverings of $M$ by geodesic balls. The supremum of multiplicities of the coverings $\Omega_j$, $j\in \nat_0$, is called the \emph{multiplicity of the sequence $\{\Omega_j\}_j$}. Furthermore, the sequence is called \emph{uniformly finite} if its multiplicity is finite and  $B(x_{j,i},2^{-j-1}r)\cap B(x_{j,k},2^{-j-1}r)=\emptyset$ for any possible $j,i,k$ with $i\neq k$.    
In \cite{Skr98} it was shown that there exists some $r_0>0$ such that for all $l\in \nat$ and $l\cdot r<r_0$ the multiplicity of the sequence $\left\{\Omega_j^{(l)}\right\}_{j=0}^{\infty}$ with $\Omega_j^{(l)}=\{B(x_{i,j},l2^{-j}r)\}_{i=0}^{\infty}$ is also finite.  

Now we recall the definitions of the atoms we use as building blocks. 

\begin{defi}
Let $s\in \real$ and $0<p\leq \infty$. Let $L$ and $K$ be integers such that $L\geq 0$ and $K\geq -1$. Let $r>0$ and $C\geq 1$ be positive constants. 
\begin{enumerate}[(i)]
\item A smooth function $a(x)$ is called an \emph{$1_L$-atom centered in $B(x,r)$} if 
\begin{align}\label{atom-1}
\supp \ a& \subset  B(x,2r),\\
\sup_{y\in M}|\nabla^k a(y)|&\leq  C\qquad\qquad \text{for any }k\leq L.
\end{align}
\item A smooth function $a(x)$ is called an \emph{$(s,p)_{L,K}$-atom centered in $B(x,r)$} if 
\begin{align}
\supp \ a&\subset B(x,2r),\\
\label{atom-4}
 \sup_{y\in M}|\nabla^k a(y)|&\leq Cr^{s-k-\frac np}, \quad \text{for any } k\leq L,\\
\label{atom-5}
 \left|\int_M a(y)\psi(y) \mathrm{d}y\right|&\leq Cr^{s+K+1+\frac{n}{p'}}\left\|\psi|C^{K+1}(B(x,2r))\right\|
\end{align}
holds for any $\psi\in C_0^{\infty}(B(x,3r))$ where $\frac{1}{p}+\frac{1}{p'}=1$. If $K=-1$, then \eqref{atom-5} follows from \eqref{atom-4}, i.e.,  no moment conditions are required. 
\item A family $\mathcal{A}_{s,p}^{L,K}$ of $1_L$-atoms and $(s,p)_{L,K}$-atoms is called a \emph{building family of atoms corresponding to the uniformly finite sequence of coverings $\left\{\Omega_j\right\}_{j=0}^{\infty}$ of $M$}, if all atoms belonging to the family are centered at the balls of the coverings $\Omega_j$, satisfy the conditions \eqref{atom-1}--\eqref{atom-5}, and if the family contains all atoms satisfying \eqref{atom-1}--\eqref{atom-5}. 
\end{enumerate}
\end{defi}

Define 
\begin{equation*}\label{sigma_p}
\sigma_p^{(n)}\define n\left(\frac 1p-1\right)_+\qquad \text{and}\qquad \sigma_{pq}^{(n)}\define n\left(\frac{1}{\min(p,q)}-1\right)_+.
\end{equation*}
Note that for $p,q\geq 1$ we always have $\sigma_p^{(n)}=\sigma_{pq}^{(n)}=0$. The atomic decomposition established in \cite{Skr98} then reads as follows. 

\begin{thm}[\textbf{Atomic decomposition}, {\cite[Thm 2]{Skr98}}]\label{thm:atomic-dec}
Let $s\in \real$ and $0<p,q\leq \infty$ ($0<p<\infty$ or $p=q=\infty$ in case of the $F$-spaces). Let $L$ and $K$ be fixed integers satisfying the conditions  
\[
L\geq \left([s]+1\right)_+ \quad \text{ and } \quad K\geq \left\{ \begin{matrix} \max\left([\sigma_{pq}^{(n)}-s],-1\right)\
                                                  &\text{for $F$-spaces},\\                                                 
\max\left([\sigma_{p}^{(n)}-s],-1\right) &\text{for $B$-spaces.}\end{matrix} \right. 
\]
There exists a positive constant $\varepsilon_0$, $0<\varepsilon_0\leq r_0$, such that there is a uniformly sequence of coverings $\{\Omega_j\}_{j=0}^{\infty}$ with $\Omega_j=\{B(x_{i,j},2^{-j}r)\}_{i=0}^{\infty}$, $r<\varepsilon_0$, and a building family of atoms corresponding to the sequence $\mathcal{A}_{s,p}^{L,K}$ with the following properties:
\begin{itemize}
\item[(i)] Each $f\in F^s_{p,q}(M)$ ($f\in B^{s}_{p,q}(M)$) can be decomposed as 
\begin{equation}\label{atom-dec}
f=\sum_{j=0}^{\infty}\sum_{i=0}^{\infty}s_{j,i}a_{j,i} 
\end{equation}
(convergence in $\mathcal{D}'(M)$), where the coefficients satisfy 
\begin{align}\nonumber
\left\|\left(\sum_{j,i=0}^{\infty}\left(s_{j,i}2^{\frac{jn}{p}}\chi_{j,i}(\cdot)\right)^q\right)^{1/q}\Bigg|L_p(M)\right\|<\infty\quad\\ 
\qquad\left(\left(\sum_{j=0}^{\infty}\left(\sum_{i=0}^{\infty}|s_{j,i}|^p\right)^{q/p}\right)^{1/q}<\infty\; \quad \text{for  $B$-spaces} \label{atom-dec2}\right) \end{align}
(with the usual modifications if $p,q=\infty$) and $\chi_{j,i}$ denotes the characteristic function of the ball $B(x_{j,i},2^{-j})$.
\item[(ii)] Conversely, suppose that $f\in \mathcal{D}'(M)$ can be represented as in \eqref{atom-dec} and \eqref{atom-dec2}. Then $f\in F^s_{p,q}(M)$ ($f\in B^s_{p,q}(M)$). Furthermore, the infimum of \eqref{atom-dec2} with respect to all admissible representations (for a fixed sequence of coverings and fixed integers $L$, $K$) is an equivalent norm in $F^s_{p,q}(M)$ ($B^s_{p,q}(M)$).  
\end{itemize}
\end{thm}

\section{Symmetries on manifolds}\label{sec:prep}

In this part we deal with symmetries on manifolds and  collect all the preliminaries necessary to understand the main results. First we introduce Besov and Triebel-Lizorkin spaces which contain G-invariant functions only  and then provide optimal atomic decompositions for them. These kind of decompositions have the advantage, that in contrast to the atomic decompositions  in Theorem \ref{thm:atomic-dec} the expression \eqref{atom-dec2}  yields an equivalent norm, i.e., the coefficients are optimal and  there is no need to take the infimum over all possible decompositions.  
Finally,  in Subsection~\ref{ass} we state an assumption on the metric in relation with  the $G$-action that is needed for our proofs later on, cf. Assumption \ref{ass:G},  and give  some first examples of manifolds satisfying this assumption. A larger class of manifolds fulfilling our requirements will be defined in Appendix~\ref{app}.

Let $(M,g)$ be a connected Riemannian manifold. Its group of isometries, denoted by $\mathrm{Isom}(M,g)$, is a Lie group with respect to the compact open topology, cf. \cite{MS39}. Let $G$ be a compact subgroup of the group of isometries $\mathrm{Isom}(M,g)$. Then $G$ is a Lie subgroup of $\mathrm{Isom}(M,g)$. From now on we study Riemannian manifolds of bounded geometry together with such a $G$-action. In particular,  the action of $G$ on $M$ is denoted by the map $(g,x)\in G\times M\mapsto g\cdot x\in M$. The \emph{orbit} of a point $x\in M$ under the $G$-action is the set $ G\cdot x\define  \{ g\cdot x\, |\, g\in G\}$ and the \emph{stabilizer} of $x$ is $S_g(x)\define \{ g\in G\ |\ g\cdot x=x\}$. For any $x\in M$ the orbit $G\cdot x$ is a smooth compact submanifold of $M$ and $G/S_G(x) \to G\cdot x$ is a diffeomorphism, \cite[Prop. 21.7]{Lee13}. 
The \emph{orbit space} $M/G$ is the set of all orbits equipped with the quotient topology. 
In general this is not a manifold but a stratified space. However,  there is always a dense subset $\tilde{M}\subset M$, the \emph{principal stratum}, such that $\tilde{M}/G$ is a manifold and $M\setminus \tilde{M}$ has measure zero. We define 
\begin{equation}\label{tilde_n}
\tilde{n}\define\dim(\tilde{M}/G)
\end{equation}
and put $\dim(M/G)=\dim(\tilde{M}/G)$, hence, $\tilde{n}$ denotes the dimension of the orbit space.   We give an easy example to illustrate the main ideas. More information can be found in Appendix \ref{app}.

\begin{ex}
Let $M=\real$ and $G=\mathbb{Z}_2$ be the discrete group reflecting every $x\in M$ at the origin. Then for the orbits we see that  $G\cdot x=\pm x$ if $x\neq 0$, $G\cdot 0=0$ (i.e., $0$ is the fix point of the group action), and for the orbit space we obtain  $M/G=[0,\infty)$. Moreover,  $\tilde{M}=\real\setminus \{0\}$ is the principal stratum and  $\tilde{M}/G=(0,\infty)$ with $\tilde{n}=\dim(\tilde{M}/G)=1$.  
\end{ex}

\subsection{{Function spaces with symmetry}}

We wish to investigate $G$-invariant functions and their properties in the sequel. For $\varphi\in \mathcal{D}(M)= C^{\infty}_c(M)$ and $g\in G$ we put $\varphi^g(x)\define \varphi(g\cdot x)$. A distribution $f\in \mathcal{D}'(M)$ is called \emph{$G$-invariant} if for any $\varphi\in C^{\infty}_c(M)$ and any $g\in G$, 
\[f(\varphi^g)=f(\varphi),\]
i.e., the function is constant along the orbits of the action. For any possible $s,p,q$ we put 
\[
R_GA^s_{p,q}(M)\define \{f\in A^s_{p,q}(M): \ f \text{ is $G$-invariant}\},  
\]
where $A\in \{B,F\}$. These spaces are closed subspaces of the spaces $B^s_{p,q}(M)$, cf. \cite[Sect. 3.1, Rem. 4]{SiSk00}.\\

\subsection{Optimal atomic decompositions adapted to the group action}\label{ssec:op}
When dealing with $G$-invariant functions we will work with (optimal) atomic decompositions (based on coverings) adapted to the action of the group $G$, as introduced by Skrzypczak in \cite{Skr03}. In this case the atomic decompositions according to Theorem~\ref{thm:atomic-dec} have some nice additional properties that will be useful later on.  

Let $A$ be a nonempty subset of a Riemannian manifold $M$ of dimension $n$. Let $\varepsilon >0$ and $\alpha\in \mN$. A subset $\mathcal{H}$ of $A$ is called an \emph{$(\varepsilon,\alpha)$-discretization of $A$} if the distance between two points in $\mathcal{H}$ is greater than or equal to $\varepsilon$ and 
\[
A\subset \bigcup_{x\in \mathcal{H}}B(x,\alpha \varepsilon). 
\]
Let us remark that for any such $A\subset M$ and for any $\epsilon>0$ there is an $(\varepsilon,1)$-discretization of $A$ which is seen immediately by chosing a maximal set of points two of them have distance greater or equal to $\epsilon$. Furthermore, if $(M,g)$ has bounded geometry and $\mathcal{H}$ is an $(\varepsilon,\alpha)$-discretization of $M$ and $m\geq \alpha$, then the family $\{B(x,m\varepsilon)\}_{x\in \mathcal{H}}$ is a uniformly locally finite covering of $M$ with multiplicity that can be estimated from above by a constant depending on $n$ and $m$, but independent of $\varepsilon$,  cf. \cite[Lem. 3]{Skr03}.
By \cite[p. 764]{Skr03} we can contruct a sequence of discretizations of $M$ adapted to $G$ in the following way:

 \begin{flushleft}
\begin{minipage}{0.47\textwidth}
 For any $j\in \mN$ there is a $(2^{-j},1)$-discretization $\{ G\cdot x_{j,k}\}_{k\in \mathcal{H}_j}$ of the principal stratum $\tilde{M}/G$. Here $\mathcal{H}_j$ is a finite set or $\mathcal{H}_j=\mN$, depending on whether $M$ is compact or not. Let now $\{x_{j,k,\ell}\}_\ell$, $\ell=1,\ldots, \mathcal{H}_{j,k}$, be a $(2^{-j},1)$-discretization of  the orbit $G\cdot x_{j,k}\subset M$ where the distances are still measured on $M$ (not with the induced metric). Then, for each $j$ the set $\{x_{j,k,l}\}_{k,l}$ is a $(2^{-j},2)$ discretization of $M$ which we call a $(2^{-j},2)$ \emph{discretization adapted to the action of the group $G$}. 
\end{minipage}\hfill \begin{minipage}{0.5\textwidth}
\captionsetup{type=figure}
\includegraphics[width=7.6cm]{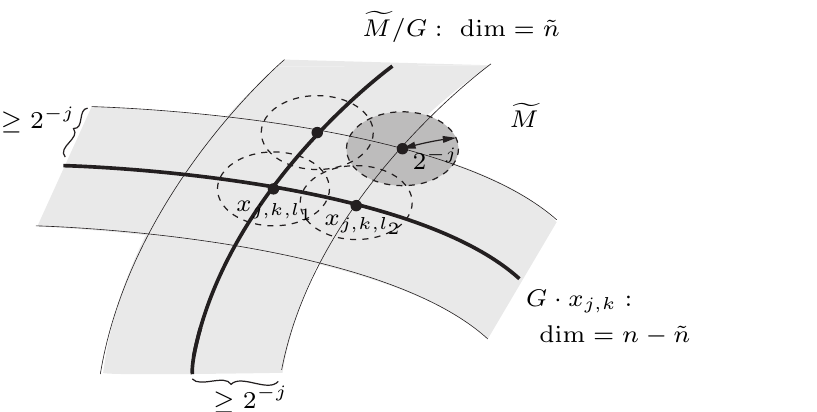}
\captionof{figure}{Locally finite covering of $M$ adapted to $G$}
\label{fig:covering}
\end{minipage}
\end{flushleft}

Furthermore, from our construction it follows  that the balls $\left\{B(x_{j,k,l},2^{-j-1})\right\}_{k,l}$ are disjoint and for each $j$ the sets  in  $\Omega_j\define \left\{B(x_{j,k,l},2^{-j+1})\right\}_{k,l}$ cover $M$.
Since $(M,g)$ has bounded geometry, the latter covers are uniformly locally finite with multiplicity bounded independent of $j$.

We use  the above covering adapted to the $G$-action together with the  heat atoms  constructed in \cite[Thm.~2]{Skr03}, which admit an optimal atomic decomposition. We briefly sketch the main ideas. Let 
\[
H_t=e^{t\Delta}, \quad t\geq 0,
\]
denote the heat semi-group, where $\Delta$ stands for the Laplace operator associated with the Riemannian metric $g$ on $M$. Let $\left\{\psi_{j,i}\right\}_{i=0}^{\infty}$ be a partition of unity corresponding to a $G$-adapted covering $\Omega_j=\{B(x_{j,i}, 2^{-j+1})\}_i$ where for simplicity we abbreviate $i=(k,\ell)$ (thus, in particular $\sum_i=\sum_{k=0}^{\mathcal{H}_j}\sum_{\ell=0}^{\mathcal{H}_{j,k}}$). 
Now we define \emph{$(s,p,r,m)$-heat atoms} of $f$ by
\[
a_{j,i}(x)\define s_{j,i}^{-1}(I-\Delta)^m\left(
\psi_{j,i}(x)\frac{1}{(r-1)!}\int_{\varepsilon b4^{-j}}^{\varepsilon b 4^{-j+1}}t^r(-\Delta)^rH_t((I-\Delta)^{-m}f)(x)\frac{\mathrm{d}t}{t}
\right)
\]
with 
\[
s_{j,i}\define 2^{j(s+2m-\frac np-2r)}\sum_{u\in I_{i}}\sup_{(t,x)\in Q_{j,u}}\left|(-\Delta)^rH_t((I-\Delta)^{-m}f)\right|(x)\in \mR
\]
where $Q_{j,i}=(b4^{-j-1},b4^{-j})\times B(x_{j,i},2^{-j})$ and $I_i=\{u\in \mN\ |\ B(x_{j,u}, 2^{-j})\cap B(x_{j,i}, 2^{-j})\neq \varnothing\}$. Here $b>1$ is a constant  and  $\varepsilon>0$ is chosen such that $\varepsilon b>1$. In particular, the constant $b$ appears in \cite[formulas (17), (18)]{Skr03} in  order to establish some estimates concerning the derivatives with respect to the time variable of $H_t f$, which in turn are immaterial for the proof of Theorem \ref{thm:optimal-dec} below.  Moreover, we define \emph{$(1,m,r)$-atoms} by 
\[
a_{0,i}(x)=s_{0,i}^{-1}(I-\Delta)^m\left(
\psi_{0,i}(x)\sum_{u=0}^{r-1}\frac{(\varepsilon b)^u}{u!}(-\Delta)^uH_{\varepsilon b}((I-\Delta)^{-m}f)\right)(x)
\] 
with 
\[
s_{0,i}=\sup_{0\leq u\leq r-1}\int_{Q_{0,i}}\left|\frac{\mathrm{d}^u}{\mathrm{d}t^u}H_s((I-\Delta)^{-m}f)(x)\right|\frac{\mathrm{d}s}{s}\mathrm{d}x.
\]

\begin{rem}\label{heat-atom-rem}
The term heat atom is justified by the fact that every $(s,p,r,m)$-heat atom is an $(s,p)_{r-2m,2m-1}$-atom, cf. \cite[p.~754]{Skr03} and the explanations given there. Note that in contrast to the general atomic decomposition in Theorem~\ref{thm:atomic-dec} the heat atoms themselves now  depend on the function one wants to decompose. But as we will see in Theorem~\ref{thm:optimal-dec} the norm is still computed using only the coefficients in this decomposition.
\end{rem}

With these preparations the atomic decomposition in Theorem~\ref{thm:atomic-dec} can now be rewritten according to our needs. The results are taken from \cite[Thm.~2]{Skr03}.

\begin{thm}[\textbf{Optimal atomic decomposition}]\label{thm:optimal-dec}
Let $s\in \real$, $1\leq p,q \leq \infty$ ($p<\infty$ or $p=q=\infty$ for $F$-spaces). Let $m$ and $r$ be nonnegative integers such that 
\begin{equation}\label{new-cond-param}
2m\geq \max([-s]+1,0)\qquad \text{and}\qquad r\geq \max([s+2m]+1,2m). 
\end{equation}
The distribution $f$ belongs to $R_GF^s_{p,q}(M)$ ($R_GB^s_{p,q}(M)$) if and only if there is an $(s,p,r,m)$-decomposition into the sum of heat atoms 
\[
f=\sum_{j=0}^{\infty}\sum_{k=0}^{\mathcal{H}_j}\sum_{l=1}^{\mathcal{H}_{j,k}}s_{j,k,l}a_{j,k,l}
\]
with $s_{j,k,l_1}=s_{j,k,l_2}=: s_{j,k}$ for $l_1,l_2\in \{0,\ldots, \mathcal{H}_{j,k}\}$ and 
\begin{equation*}
\|f|F^s_{p,q}(M)\|\sim \left\|\left(\sum_{j=0}^{\infty}2^{j\frac{nq}{p}}\left(\sum_{k=0}^{\mathcal{H}_j}\sum_{l=1}^{\mathcal{H}_{j,k}}|s_{j,k}|\chi_{j,k,l}(\cdot)\right)^q\right)^{1/q}\Bigg|L_p(M)\right\|<\infty 
\end{equation*}
for $F$-spaces, where $\chi_{j,k,l}$ denotes the characteristic function of $B(x_{j,k,l},2^{-j})$ and 
\begin{equation*}
\|f|B^s_{p,q}(M)\|\sim \left(\sum_{j=0}^{\infty}\left(\sum_{k=0}^{\mathcal{H}_j}\mathcal{H}_{j,k}|s_{j,k}|^p\right)^{q/p}\right)^{1/q}<\infty
\end{equation*}
for $B$-spaces (with the usual modifications if $p,q=\infty$). 
\end{thm}

\begin{rem}  Compared to \cite{Skr03} we changed the conditions on the parameters $m$ and $k$ in \eqref{new-cond-param}. In \cite[Thm.~2]{Skr03} it was assumed that 
\begin{equation}\label{cond-param-old}
m\geq \max([-s]+1,0)\qquad \text{and}\qquad k\geq [s+2m]+1. 
\end{equation}
By Remark \ref{heat-atom-rem} every $(s,p,k,m)$-heat atom is an $(s,p)_{k-2m,2m-1}$-atom, i.e., we have the relation  $L=k-2m$ and $K=2m-1$.  However, for $s=-2$ from \eqref{cond-param-old} we get $m\geq 3$, thus, $s+2m\geq 0$. This leads to the following condition on the parameter $L$, 
\[
L=k-2m\geq [s+2m]+1-2m=[s]+1=-1,
\]
which is a contradiction to Theorem \ref{thm:atomic-dec} where $L\geq \max([s]+1,0)$ is required. We remark that already for atomic decompositions in the euclidean case  the imposed moment conditions, i.e., the restrictions on $L$, are immaterial and refer to \cite{Sch09, SV09} in this context.  Our new conditions \eqref{new-cond-param} result from the calculations 
\[
K=2m-1\geq \max([-s],-1)\quad \iff \quad 2m\geq \max([-s]+1,0)
\]
and 
\[
L=k-2m\geq \max([s]+1,0)\quad \iff \quad k\geq \max([s+2m]+1, 2m), 
\]
guaranteeing that the conditions imposed on $L$ and $K$ in Theorem \ref{thm:atomic-dec} are satisfied. 
\end{rem}

We will later apply Theorem~\ref{thm:optimal-dec} to obtain both decay properties of $G$-invariant functions near infinity and near the boundary of the main stratum and Hölder regularity inside the main stratum, cf. Theorems \ref{thm-strauss} and  \ref{hoelder-reg}, respectively.

\subsection{Assumption on the Riemannian metric related to the \texorpdfstring{$G$}{G}-action.}\label{ass}

In this part we collect the  assumptions we shall need in order to prove our results in Sections~\ref{sec:4} and~\ref{sec:5}. Additionally, we  provide  some easy examples satisfying our requirements. More on this can be found in Appendix~\ref{app}.

The assumption we need will be on the $G$-adapted covering presented in  Section~\ref{ssec:op}. Some preparations are needed beforehand. 

\begin{minipage}{0.6\textwidth}
We cover the orbit $G\cdot x_{j,k}$ with balls of radius $2\cdot 2^{-j}$ and put  
\begin{equation}\label{Pjk}
P_{j,k}\define \bigcup_{l=1}^{\mathcal{H}_{j,k}}B(x_{j,k,l}, 2^{-j+1}).
\end{equation}
By our construction, for each $x_{j,k_j}$ it is possible to find some $x_{j+1,k_{j+1}}$ such that the induced distance on $\tilde{M}/G$ is 
$\text{dist}_{M/G} (x_{j,k_j},x_{j+1,k_{j+1}})\leq 2^{-j}$ from which we deduce $B_{M/G}(x_{j+1,k_{j+1}},2^{-j})\subset B_{M/G}(x_{j,k_j}, 2^{-j+1})$. This leads  to 
\[
P_{j+1,k_{j+1}}\subset P_{j,k_j}. 
\]
\end{minipage}\hfill \begin{minipage}{0.29\textwidth}
\begin{psfrags}
\includegraphics[width=4cm]{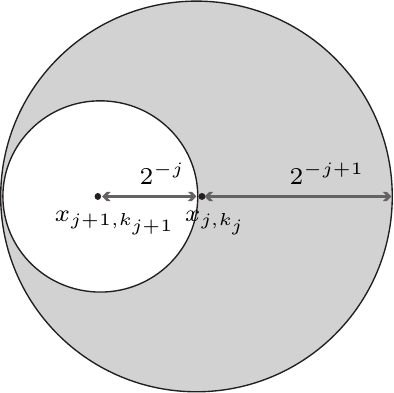}
\end{psfrags}
\end{minipage}\\

\begin{assumption}[\textbf{On the $G$-action}] \label{ass:G} There is a $G$-adapted covering, cf. Figure~\ref{fig:covering} on page~\pageref{fig:covering} and the explanations given aside, such that the following holds. 
 \begin{enumerate}[(i)]
  \item For all $x\in B(x_{j,k,\ell}, 2^{-j+1})$ we have 
  \begin{align}\label{est-Hjk} \mathcal{H}_{j,k} \gtrsim  \text{vol}_{n-\tilde{n}}(G\cdot x) 2^{j(n-\tilde{n})},\end{align}
 where $\tilde{n}$ denotes the dimension of the orbit space, cf. \eqref{tilde_n}. 
  \item There are constants $b>0$ and $\epsilon>0$  such that for all $x\in M$ with $\text{dist}_g(x, M\setminus \tilde{M})\geq r$ it holds 
  \begin{align} \label{est-volGx} \text{vol}_{n-\tilde{n}}(G\cdot x) \gtrsim \min\{ r^b, \epsilon\}.  \end{align}
  \item We have the estimate \begin{align}\label{est-Pjk}\vol(P_{j,k_j}\setminus P_{j+1,k_{j+1}})\gtrsim 2^{-jn}\mathcal{H}_{j,k_j},\end{align}
 \end{enumerate}
where the constants do not depend on $j$, $x$, $r$ (and for $(iii)$ not on the chosen $k_j$).
\end{assumption}

These assumptions are quite technical. We give two easy examples below. For a more detailed discussion we refer to Appendix~\ref{app}.

\begin{ex}\label{ex:rn}
 Let $SO(n)$ act on $\mR^n$ by rotation around the origin. For this action the orbit space is  $[0,\infty)$ if  $\tilde{n}\define \dim (\tilde{M}/G)=1$. Let $\{y_{k}^r\}_{k\in \mathcal{K}_r}$ be a $(r,1)$ discretization of the standard sphere $S^{n-1}$. Then by standard comparison of volumes we see that $|\mathcal{K}_r|\sim r^{1-n}$. Then $B(y_k^{2^{-j}}, 2^{-j+1})$  cover $\{x\in \mR^n\ |\ \text{dist}(x, S^{n-1})\leq 2^{-j}\}$. Thus, by the scaling invariance of $\mR^n$ the set of  the points $x_{j,k,l}= 2^{-k}y_l^{2^{-j}}$  gives a $G$-adapted covering. In particular, we have $\mathcal{H}_{j,k}\sim 2^{-(j-k)(1-n)}$ and $\vol_{n-1}(SO(n)\cdot x)\sim |x|^{n-1}$. Thus, for this covering Assumption~\ref{ass:G} is obviously true with $b=n-1$.\end{ex}

\begin{ex}
We consider $M=\mR^4\cong \mC\times \mC$ with $G=S^1\times S^1\cong SO(2)\times SO(2)\subset SO(4)$ acting on $M$ as follows
\[
G\times M\to M,\quad ((\alpha, \beta),(z_1,z_2))\mapsto (e^{\i \alpha} z_1, e^{\i \beta} z_2).
\]
\begin{minipage}{0.4\textwidth}
\begin{psfrags}
\includegraphics[width=6cm]{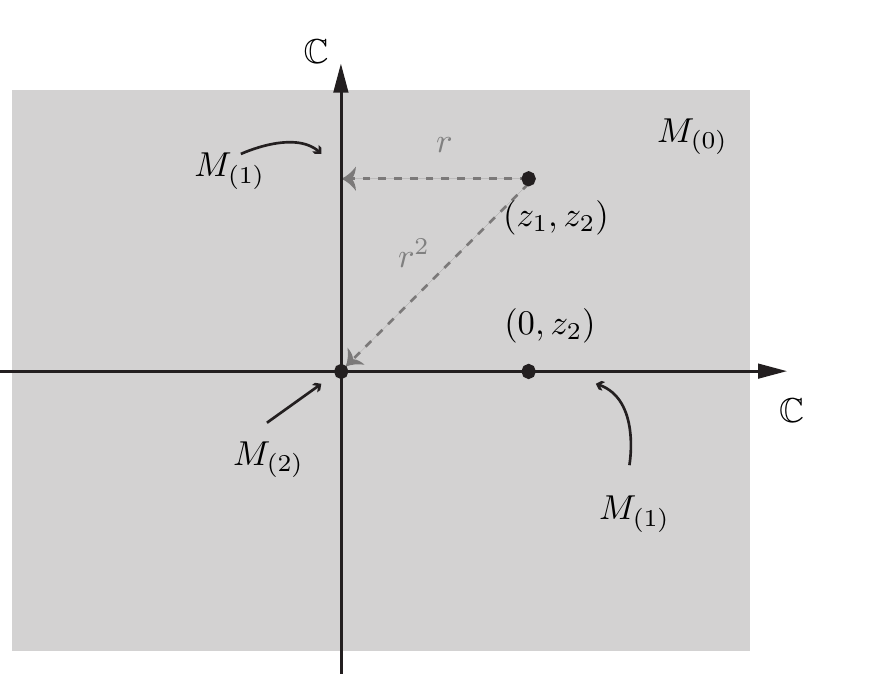}
\end{psfrags}
\end{minipage}\hfill \begin{minipage}{0.55\textwidth}
Our manifold stratifies as follows \[M=M_{(0)}\cup M_{(1)}\cup M_{(2)}\]
where $M_{(1)}=(\mathbb{C}\setminus \{0\}\times \{0\})\cup (\{0\}\times \mathbb{C}\setminus \{0\}) $ and  $M_{(2)}=\{(0,0)\}$
according to the following orbits: \\[0.3cm]
\renewcommand{\arraystretch}{1.2}
\begin{tabular}{c|c|c}
Strata & Orbit(s) & $\dim(\text{Orbit(s)})$\\
\hline
$\tilde{M}=M_{(0)}$& $S^1\times S^1$& $2$\\
\hline
$M_{(1)}$& $\{0\}\times S^1$, $S^1\times \{0\}$& $1$\\
\hline
$M_{(2)}$& $\{0\}$& $0$
\end{tabular}\\[0.3cm]

\end{minipage}\\
As for the volumes of the orbits we see that 
\begin{align*}
\vol_2(G\cdot (z_1,z_2))&=& (2\pi)^2|z_1|\cdot |z_2|
\sim \begin{cases}
r,&  z_2 \text{ fixed}, z_1=r\rightarrow 0,\\
r^2,& z_1=z_2=r \rightarrow 0\\
\end{cases}\\
\vol_1(G\cdot (z,0))&=& 2\pi |z_1|\sim r \quad z=r\rightarrow 0.
\end{align*}
Let $x_{j,k}\in \tilde{M}$. For $b_l=\dim(G\cdot x_{j,k})-\min\{\dim(G\cdot x)| \ x\in M_{(l)}, \ l=1,2\}$ we calculate 
\[
b=b_2=2-0=2, \qquad b_1=2-1=1.
\]
This also provides a counterexample demonstrating that the 'optimal choice' for $b$  in  \cite[p.~275/276]{Skr02} is not correct.  There it was stated that for $M=\rn$ and $G$ being a closed subgroup of $O(n)$ we should get $b=\min_{x\in S^{n-1}}\dim (G\cdot x)=1$, which is obviously false by the above considerations.  
\end{ex}

\section{Decay properties and local smoothness of \texorpdfstring{$G$}{G}-invariant functions}\label{sec:4}

In this section we state our main results concerning decay properties and local smoothness of $G$-invariant function. Our proofs heavily rely on the optimal atomic decomposition adapted to the $G$-action from Theorem \ref{thm:optimal-dec}.

\begin{thm}\label{thm-strauss}
 Let $(M^n,g)$ be a Riemannian manifold of bounded geometry and let $G$ be a compact subgroup of the group of isometries  $\mathrm{Isom}(M,g)$. We assume that the $G$-action on $M$ fulfills Assumptions~\eqref{est-Hjk} and~\eqref{est-Pjk}. Moreover, we put  $\tilde{n}\define \dim (\tilde{M}/G)$.
 \begin{itemize}
 \item[(i)] Let $1\leq p,q\leq \infty$ and  $s> \frac{\tilde{n}}{p}$ or ($s= \frac{\tilde{n}}{p}$ and $q=1$). Then, there is a constant $C>0$ such that 
\[ |f(x)|\leq C \vol_{n-\tilde{n}}(G\cdot x)^{-\frac{1}{p}} \Vert f\, |\, B^{s}_{p,q}(M)\Vert \]
holds for all $f\in R_GB^{s}_{p,q}(M)$ and $x\in \tilde{M}$.  
\item[(ii)] { Let $1\leq p<\infty$, $1\leq q\leq \infty$, and $s> \frac{\tilde{n}}{p}$ or ($s= \frac{\tilde{n}}{p}$ and $p=1$). Then, there is a constant $C>0$ such that 
\[ |f(x)|\leq C \vol_{n-\tilde{n}} (G\cdot x)^{-\frac{1}{p}} \Vert f\, |\, F^{s}_{p,q}(M)\Vert \]
holds for all $f\in R_GF^{s}_{p,q}(M)$ and $x\in \tilde{M}$. }
\end{itemize}
\end{thm}

\begin{proof} \underline{Step 1:} We prove (i).  Let $x\in \tilde{M}$ and $f\in R_GB^s_{p,q}(M)$. By Theorem~\ref{thm:optimal-dec} there exists an atomic decomposition 
 \[
f=\sum_{j=0}^{\infty}\sum_{k=0}^{\mathcal{H}_{j}}\sum_{l=0}^{\mathcal{H}_{j,k}}s_{j,k}a_{j,k,l}
\quad \text{\ such that\ }\quad 
\left(\sum_{j=0}^{\infty}\left(\sum_{k=0}^{\mathcal{H}_j}\mathcal{H}_{j,k}|s_{j,k}|^p\right)^{q/p}\right)^{1/q}<\infty
 \]
 (with the usual modifications if $p=\infty$ or $q=\infty$).
 Then the main part of $f$ near $x$ is given by the function 
 \begin{equation}\label{main_f}
f^M(y)=\sum_{j=0}^{\infty}s_{j,k_j}a_{j,k_j,l_j}(y)
 \end{equation}
 where $k_j$, $l_j$ are chosen such that $\text{dist}_g(x,x_{j,k_j,l_j})$ is minimal. Then in particular $x_{j,k_j,\ell_j}\in B(x, 2^{-j+1})$ since $x\in \supp \ a_{j,k_j,l_j}\subset B(x_{j,k_j,l_j}, 2^{-j+1})$. In fact, $f$ is  a finite sum of functions of the above type, cf. Figure \ref{fig:covering} on page \pageref{fig:covering} and the explanations given there. 
 Without loss of generality we give an estimate for the main part of $f$ only. 
 We use the estimate  \eqref{est-Hjk}, i.e., $\mathcal{H}_{j,k}\gtrsim \vol_{n-\tilde{n}} (G\cdot x) 2^{j(n-\tilde{n})}$ for all 
 $x\in B(x_{j,k,\ell}, 2^{-j+1})$.  Now together with
the normalization of atoms, see \eqref{atom-4}, we get 
 \begin{align}
 |f^M(y)|
 	&\leq  \sum_{j=0}^{\infty}2^{-j(s-\frac np)}|s_{j,k_j}|\notag\\
 	&= \sum_{j=0}^{\infty}2^{-j(s-\frac np)}\frac{2^{\frac{-j(n-\tilde{n})}{p}}}{\vol_{n-\tilde{n}} (G\cdot x)^{1/p}}\cdot \vol_{n-\tilde{n}} (G\cdot x)^{1/p}2^{\frac{j(n-\tilde{n})}{p}}|s_{j,k_j}|\notag\\
 	&\lesssim  \vol_{n-\tilde{n}} (G\cdot x)^{-1/p}\sum_{j=0}^{\infty}2^{-j(s-\frac {\tilde{n}}{p})}\cdot \mathcal{H}_{j,k_j}^{1/p}|s_{j,k_j}|\notag\\
 	&\leq  \vol_{n-\tilde{n}} (G\cdot x)^{-1/p}\sum_{j=0}^{\infty}2^{-j(s-\frac {\tilde{n}}{p})}\left(\sum_{k=0}^{\mathcal{H}_{j}}\mathcal{H}_{j,k} |s_{j,k}|^p\right)^{1/p}\notag\\
 	&\lesssim  \vol_{n-\tilde{n}} (G\cdot x)^{-1/p}\|f|B^s_{p,q}(M)\|,\label{help-1}  
 \end{align}
 where in the last line we used H\"older's inequality if  $s>\frac{\tilde{n}}{p}$ (and nothing for  $s=\frac{\tilde{n}}{p}$ and $q= 1$). The suppressed constants do not depend on $f$ and $x$.  \\
 \underline{Step 2:} We turn towards (ii). 
Let $s>\frac{\tilde{n}}{p}$, $1\leq p<\infty$ or ($s=\frac{\tilde{n}}{p}$ and $p=1$). We deal with the spaces with  $q=\infty$ only, since the rest follows from the embedding 
 $F^s_{p,q}(M)\hookrightarrow  F^{s}_{p,\infty}(M)$. 
 Again, we consider the main part of $f$ according to \eqref{main_f}. Using the very same calculation as in \eqref{help-1}, except the last two lines, we have 
 \begin{align*}
 |f^M(x)|& \leq  \vol_{n-\tilde{n}} (G\cdot x)^{-1/p}\sum_{j=0}^{\infty}2^{-j(s-\frac{\tilde{n}}{p})}\mathcal{H}_{j,k_j}^{1/p}|s_{j,k_j}|
 \end{align*}
With Hölder's inequality if $s>\frac{\tilde{n}}{p}$ (and nothing for $s=\frac{\tilde{n}}{p}$ and $p=1$) we obtain
\begin{align}
  |f^M(x)| &\leq   \vol_{n-\tilde{n}} (G\cdot x)^{-1/p}\left(\sum_{j=0}^{\infty}\mathcal{H}_{j,k_j}|s_{j,k_j}|^p\right)^{1/p}\notag\\
 &\leq  \vol_{n-\tilde{n}} (G\cdot x)^{-1/p}\left(\sum_{j=0}^{\infty}\sup_{i=0,\ldots,j}\left(|s_{i,k_i}|^p2^{in}\right)2^{-jn}\mathcal{H}_{j,k_j}\right)^{1/p}.\label{help1}
 \end{align}
  Let now  $\tilde{\chi}_{j,k}$ denote the characteristic function of $P_{j,k}=\bigcup_{l=1}^{\mathcal{H}_{j,k}}B(x_{j,k,l}, 2^{-j+1})$, cf. \eqref{Pjk}.  We shall proceed with 
 \begin{align}
 \sum_{j=0}^{\infty}&\sup_{i=0,\ldots, j}2^{i\frac np}|s_{i,k_i}|\left(\tilde{\chi}_{j,k_j}(\cdot)-\tilde{\chi}_{j+1,k_{j+1}}(\cdot)\right)\notag\\
 &= \sup_{i\in \mathbb N_0}\sum_{j=i}^{\infty}2^{i\frac np}|s_{i,k_i}|\left(\tilde{\chi}_{j,k_j}(\cdot)-\tilde{\chi}_{j+1,k_{j+1}}(\cdot)\right) = \sup_{i\in \mathbb N_0}2^{i\frac np}|s_{i,k_i}|\tilde{\chi}_{i,k_i}(\cdot),\label{help1a}
 \end{align}
 where we use the fact that  we have a telescopic sum. Now taking the $L_p$-norm on both sides and using the fact that $P_{j,k_j}\setminus P_{j+1,k_{j+1}}$ are disjoint sets for different values of $j$ yields 
 \begin{align}
 \left\|\sup_{i\in \mathbb N_0} \right.& \left. 2^{i\frac np}|s_{i,k_i}|{\tilde{\chi}}_{i,k_i}(\cdot)\bigg|L_p(M)\right\|^p\notag \\
 &= \left\|\sum_{j=0}^{\infty}\sup_{i=0,\ldots, j}2^{i\frac np}|s_{i,k_i}|\left(\tilde{\chi}_{j,k_j}(\cdot)-\tilde{\chi}_{j+1,k_{j+1}}(\cdot)\right)\bigg|L_p(M)\right\|^{p}\notag\\
 &= \sum_{j=0}^{\infty}\sup_{i=0,\ldots, j}\left(2^{jn}|s_{i,k_i}|^p\right)\vol(P_{j,k_j}\setminus P_{j+1,k_{j+1}})\notag\\
 &\gtrsim \sum_{j=0}^{\infty}\sup_{i=0,\ldots, j}\left(2^{jn}|s_{i,k_i}|^p\right) 2^{-jn}\mathcal{H}_{j,k_j},\label{help2}
 \end{align}
where in the last step we used assumption~\eqref{est-Pjk}. 
 Now \eqref{help1} and \eqref{help2} yield 
 \begin{align}
 |f^M(x)|
 &\leq  \vol_{n-\tilde{n}} (G\cdot x)^{-1/p} \left\|\sup_{i\in \mathbb N_0}2^{i \frac np}|s_{i,k_i}|\tilde{\chi}_{i,k_i}(\cdot)\bigg|L_p(M)
 \right\|\notag\\
 &\lesssim \vol_{n-\tilde{n}} (G\cdot x)^{-1/p} \left\|\sup_{i\in \mathbb N_0}2^{i \frac np}|s_{i,k_i}|\sum_{l=1}^{\mathcal{H}_{i,k_i}}{\chi}_{i,k_i,l}(\cdot)\bigg|L_p(M)
\right\|\notag\\
 &\lesssim vol_{n-\tilde{n}} (G\cdot x)^{-1/p} \left\|\sup_{i\in \mathbb N_0}2^{i \frac np}\sum_{k=0}^{\mathcal{H}_i}\sum_{l=1}^{\mathcal{H}_{i,k_i}}|s_{i,k}|{\chi}_{i,k,l}(\cdot)\bigg|L_p(M)\right\|\notag\\
& \sim \vol_{n-\tilde{n}} (G\cdot x)^{-1/p} \|f|F^{s}_{p,\infty}(M)\|, \label{help-F}
 \end{align}
 which completes the proof. 
\end{proof}

\begin{rem}
The restriction $1\leq p,q\leq \infty$ in Theorem~\ref{thm-strauss} comes from the fact that we need to use the atomic decomposition via heat atoms from Theorem~\ref{thm:optimal-dec}, which only works for this range of parameters. 
\end{rem}

Before studying H\"older regularity of  $G$-invariant functions, we add the following observation about the convergence of the partial sums of our atomic decompositions. Let $f\in R_GA^s_{p,q}(M)$, $A\in \{B,F\}$,  have an optimal atomic decomposition 
\[f=\sum_{j=0}^{\infty}\sum_{k=0}^{\mathcal{H}_{j}}\sum_{l=1}^{\mathcal{H}_{j,k}}s_{j,k}a_{j,k,l}\]
as provided by Theorem~\ref{thm:optimal-dec}. Consider the partial sums
\[S^Jf=\sum_{j=0}^{J}\sum_{k=0}^{\mathcal{H}_j}\sum_{l=1}^{\mathcal{H}_{j,k}}s_{j,k}a_{j,k,l}.\] 

\begin{lem}\label{lem-trace} 
Let $1\leq p<\infty$ and $1\leq q\leq \infty$. Suppose either $s>\frac{\tilde{n}}{p}$  or  $s=\frac{\tilde{n}}{p}$ with  $q=1$ for B-spaces. For F-spaces we assume  $s>\frac{\tilde{n}}{p}$  or  $s=\frac{\tilde{n}}{p}$ with $p=1$.  Then $f\in R_GA^s_{p,q}(M)$, $A\in \{B,F\}$,  implies $f$ is continuous in $\tilde{M}$ and 
\[
\lim_{J\rightarrow \infty}\|f-S^Jf\ \big|\ L_{\infty}(\tilde{M})\|=0,
\]
i.e., we have uniform convergence of the partial sums on $\tilde{M}$. 
\end{lem}

\begin{proof} Concerning B-spaces, thanks to $f\in R_GB^s_{p,q}(M)$ we know that  for each $\delta>0$ there exists a number $J_0$ such that 
\[
\left(\sum_{j=J+1}^{\infty}\left(\sum_{k=0}^{\mathcal{H}_j}\mathcal{H}_{j,k}|s_{j,k}|^p\right)^{q/p}\right)^{1/q}<\delta,  
\]
for all $J\geq J_0$. Consider $x\in \tilde{M}$.  Then as in \eqref{help-1} we get 
\begin{equation}\label{help-B}
|f^M(x)-S^Jf^M(x)|=\left|\sum_{j=J+1}^{\infty}s_{j,k_j}a_{j,m_j,l_j}(x)\right|\lesssim \vol_{n-\tilde{n}} (G\cdot x)^{-1/p}\delta,
\end{equation}
for all $J\geq J_0$, where the suppressed constants are independent of $x$ and $\delta$. Hence, we have convergence of $S^Jf$ in the uniform norm and for that reason the limit itself is a continuous function. The proof for F-spaces follows along the same lines, using that if $f\in R_GF^s_{p,q}(M)$ for each $\delta>0$ there exists a number $J_0$ such that 
\[
 \left\|\left(\sum_{j=J+1}^{\infty}2^{j\frac{nq}{p}}\left(\sum_{k=0}^{\mathcal{H}_j}\sum_{l=1}^{\mathcal{H}_{j,k}}|s_{j,k}|\chi_{j,k,l}(\cdot)\right)^q\right)^{1/q}\Bigg|L_p(M)\right\|<\delta, 
\]
for all $J\geq J_0$ 
and then \eqref{help-F} instead of \eqref{help-1} in \eqref{help-B}. In particular, the supremum in \eqref{help-F}  now is $\displaystyle  \sup_{i=J+1,\ldots, \infty}$, since  in \eqref{help1} we now have  $\displaystyle  \sum_{j=J+1}^{\infty}\sup_{i=J+1,\ldots, j}=\sup_{i=J+1,\ldots,\infty}\sum_{j=i, \ldots, \infty}$ in \eqref{help1a}. Thus, in \eqref{help-F} we end up with $\displaystyle  \sup_{i=J+1,\ldots, \infty}$.  
This proves the lemma.  
\end{proof}

We now investigate H\"older regularity in the main stratum of $G$-invariant functions. In particular,  this will show that Lemma \ref{lem-trace},  which gives continuity of $f$ on $\tilde{M}$,  can be strengthened. \\

\begin{thm}[\textbf{H\"older regularity away from the singular strata}]\label{hoelder-reg}  Let $(M^n,g)$ be a Riemannian manifold of bounded geometry and let $G$ be a compact subgroup of  $\mathrm{Isom}(M,g)$. We assume that the $G$-action on $M$ fulfills assumptions~\eqref{est-Hjk} and~\eqref{est-volGx}. Moreover, let   $1\leq p,q\leq \infty$ ($p<\infty$ for $F$-spaces), and   $s>\frac{\tilde{n}}{p}$. Furthermore, let $f\in R_GA^s_{p,q}(M)$ with $A\in \{B,F\}$ such that $\supp f\subset \tilde{M}$. Then $f$ belongs to the H\"older-Zygmund space 
\begin{equation*}
{ f}\in \mathcal{C}^{s_0}(\tilde{M}), \qquad s_0=s-\frac{\tilde{n}}{p},
\end{equation*}
defined in \eqref{def-zygmund-M}. 
Moreover, there is an $r_0>0$ such that 
\begin{equation*}
\|{f}|\mathcal{C}^{s_0}({M})\|\gtrsim r^{-b/p}\|f|A^s_{p,q}(M)\|
\end{equation*}
for all $f$ with $\supp f\subset \tilde{M}_r=\{x\in M\ |\ \mathrm{dist}(x, M\setminus\tilde{M})\geq r\}$ for $r<r_0$.  The constant of the inequality depends on $r_0$ but not on $r$ and  $b$ is the constant from Assumption~\eqref{est-volGx}.
\end{thm}

\begin{proof} \underline{\emph{Step 1:}} The case $p=\infty$ is clear observing that $\mathcal{C}^{s_0}(M)=B^{s_0}_{\infty,\infty}(M)$ in terms of equivalent quasi-norms and using the elementary embedding 
\[
B^{s_0}_{\infty,q}(M)\subset B^{s_0}_{\infty,\infty}(M)\quad \text{for all }q. 
\]
\underline{\emph{Step 2:}} Let $p<\infty$ and consider $B$-spaces. Let $f\in R_GB^s_{p,q}(M)$ for some $s>\tilde{n}/p$. According to Theorem~\ref{thm:optimal-dec}, $f$ has the following optimal atomic decomposition
 \begin{equation}\label{dec-adapted}
f(x)=\sum_{j=0}^{\infty}\sum_{k=0}^{\mathcal{H}_j}\sum_{l=1}^{{\mathcal{H}}_{j,k}}s_{j,k}a_{j,k,l}(x)
 \end{equation}
with $a_{j,k,l}$ being an $(s,p,r,m)$-atom with $m=0$ and $r\geq [s]+1$ and 
\[
\|f|B^s_{p,q}(M)\|\sim \left(\sum_{j=0}^{\infty}\Bigg(\sum_{k=0}^{{\mathcal{H}}_j}\sum_{l=1}^{{\mathcal{H}}_{j,k}}|s_{j,k}|^p\Bigg)^{q/p}\right)^{1/q}.
\]
Put 
\begin{equation}\label{trace-0}
b_{j,k,l}(x)\define 2^{-j\left(\frac{n-\tilde{n}}{p}\right)}a_{j,k,l}(x).
\end{equation}
Since $a_{j,k,l}$ is an $(s,p,r,m)$-heat atom with $m=0$, from Remark \ref{heat-atom-rem} we deduce that it is also an 
 $(s,p)_{r,-1}$-atom, thus,   the functions $b_{j,k,l}(x)$ are $(s-\frac{\tilde{n}}{p},\infty)_{r,-1}$-atoms on $M$: 
\begin{itemize}
\item[$\bullet$] $\supp\, b_{j,k,l}=\supp\, a_{j,k,l}$\\
\item[$\bullet$] $\sup |D^{\alpha}b_{j,k,l}|=\sup \left|D^{\alpha}\left(2^{-j\left(\frac{n-\tilde{n}}{p}\right)}a_{j,k,l}\right)\right|\leq 2^{-j\left(\frac{n-\tilde{n}}{p}+s-\frac np-|\alpha|\right)}=2^{-j(s-\frac{\tilde{n}}{p}-|\alpha|)}$\\
\item[$\bullet$]  No moment conditions are needed since  $K=-1.$ 
\end{itemize}
 Since $s>s_0\define s-\frac{\tilde{n}}{p}>0$  we deduce $r\geq [s_0]+1$. It follows from Theorem \ref{thm:atomic-dec} that 
 \begin{equation}\label{trace-1}
f(x)=\sum_{j=0}^{\infty}\sum_{k=0}^{\mathcal{H}_j}\sum_{l=1}^{\mathcal{H}_{j,k}}s_{j,k}2^{j\frac{(n-\tilde{n})}{p}}b_{j,k,l}(x)
 \end{equation}
 satisfies
 \begin{align*}
 \|f|B^{s-\frac{\tilde{n}}{p}}_{\infty,q}(M)\|
 &\lesssim  \left(\sum_{j=0}^{\infty}\left(\sup_{\mtop{k=0,\ldots, \mathcal{H}_j,}{l=1,\ldots, \mathcal{H}_{j,k}}}|s_{j,k}|2^{j\frac{n-\tilde{n}}{p}}\right)^q\right)^{1/q}.
 \end{align*}
 From Assumptions~\eqref{est-Hjk} and~\eqref{est-volGx} we have $\mathcal{H}_{j,k} \gtrsim  \min\{ r^b, \epsilon\} 2^{j(n-\tilde{n})}$ for $\text{dist}_g(x_{j,k,\ell}, M\setminus \tilde{M})\geq r$. Choose $r_0= \epsilon^{1/b}$ and let $r\leq r_0$. Then
 
 \begin{align*}
  \|f|B^{s-\frac{\tilde{n}}{p}}_{\infty,q}(M)\| 
 &\lesssim  \left(\sum_{j=0}^{\infty}\left(\sup_{k=0,\ldots, \mathcal{H}_j}|s_{j,k}|\mathcal{H}_{j,k}^{1/p}r^{-b/p}\right)^q\right)^{1/q}\\
 &\leq  r^{-b/p}\left(\sum_{j=0}^{\infty}
 \left(\sum_{k=0}^{\mathcal{H}_j}
 \left(|s_{j,k}|\mathcal{H}_{j,k}^{1/p}\right)^p
 \right)^{q/p}\right)^{1/q}\\
 &\leq  r^{-b/p}\left(\sum_{j=0}^{\infty}\left(\sum_{k=0}^{\mathcal{H}_j}\sum_{l=1}^{\mathcal{H}_{j,k}}|s_{j,k}|^p\right)^{q/p}\right)^{1/q}\sim r^{-b/p}\|f|B^{s}_{p,q}(M)\|,\\
 \end{align*}
 where the suppressed constants do not depend on $f$.  The fact that  $f\in \mathcal{C}^{s_0}(M)$ for $s_0=s-\frac{\tilde{n}}{p}$ is now a consequence of Step 1. \\ 
\underline{\emph{Step 3:}} Let $p<\infty$ and consider the $F$-spaces.  Again, $f\in R_GF^s_{p,q}(M)$ can be decomposed as in \eqref{dec-adapted} with 
\[\|f|F^s_{p,q}(M)\|\sim \left\|\left(\sum_{j=0}^{\infty}2^{j\frac{nq}{p}}\left(\sum_{k=0}^{\mathcal{H}_j}\sum_{l=1}^{\mathcal{H}_{j,k}}|s_{j,k}|\chi_{j,k,l}(\cdot)\right)^q\right)^{1/q}\Bigg|L_p(M)\right\|.\]
Again using the rescaled decomposition from \eqref{trace-0} and \eqref{trace-1} we get 
\begin{align}
\|f|B^{s-\frac{\tilde{n}}{p}}_{\infty,\infty}(M)\|
&= \sup_{j\in \mathbb N_0}\sup_{\mtop{k=0,\ldots, \mathcal{H}_j}{l=1,\ldots, \mathcal{H}_{j,k}}}|s_{j,k}|2^{j\frac{(n-\tilde{n})}{p}}\notag\\
&\lesssim  r^{-b/p}\sup_{j\in \mathbb N_0}\sup_{{k=0,\ldots, \mathcal{H}_j}}
\mathcal{H}_{j,k}^{1/p}|s_{j,k}|.\label{hoelder-est-1}
\end{align}
Now using that $\chi_{j,k,l}$ is the characteristic function of $B(x_{j,k,l},2^{-j})$, our covering is locally finite, and  the bounded geometry of $M$ implies  $\vol  B(x,2^{-j}) \gtrsim 2^{-jn}$, \cite[Lem. 2.2 and below]{Heb99}, we obtain
\begin{align*}
\mathcal{H}_{j,k}^{1/p}
= 2^{jn/p}\left(2^{-jn}\mathcal{H}_{j,k}\right)^{1/p}\lesssim   \left\|2^{jn/p}\sum_{l=1}^{\mathcal{H}_{j,k}}\chi_{j,k,l}(\cdot)\Bigg| L_p(M)\right\|.
\end{align*}
Inserting this estimate in \eqref{hoelder-est-1},  gives 
\begin{align*}
\|f|B^{s-\frac{\tilde{n}}{p}}_{\infty,\infty}(M)\|
&\leq   r^{-b/p}\left\|\sup_{j\in \mathbb N_0}2^{jn/p}\sup_{{k=0,\ldots, \mathcal{H}_j}}
\sum_{l=1}^{\mathcal{H}_{j,k}}|s_{j,k}|\chi_{j,k,l}(\cdot)\Bigg| L_p(M)\right\|\\
&\leq  r^{-b/p}\left\|\left(\sum_{j=0}^{\infty}2^{jnq/p}\left(\sup_{{k=0,\ldots, \mathcal{H}_j}}
\sum_{l=1}^{\mathcal{H}_{j,k}}|s_{j,k}|\chi_{j,k,l}(\cdot)\right)^q\right)^{1/q}\Bigg| L_p(M)\right\|\\
&\leq  r^{-b/p}\left\|\left(\sum_{j=0}^{\infty}2^{jnq/p}\left(\sum_{k=0}^{\mathcal{H}_j}
\sum_{l=1}^{\mathcal{H}_{j,k}}|s_{j,k}|\chi_{j,k,l}(\cdot)\right)^q\right)^{1/q}\Bigg| L_p(M)\right\|\\
&=r^{-b/p}\|f|F^s_{p,q}(M)\|,
\end{align*}
where we used the fact that $l_q\hookrightarrow l_{\infty}$  in the second  step. This finally completes the proof. 
\end{proof}

\begin{rem} Theorem~\ref{hoelder-reg} concerning the H\"older regularity on $\tilde{M}$ is a generalization of \cite[Th.~5]{Skr02}. Our proof is much simpler. In particular, here we do not rely on trace results on the orbit space $\tilde{M}/G$ and the orbits $G\cdot x$ as was done there.   
However, we wish to remark that the proof concerning traces on the orbits presented in \cite[L.~5]{Skr02}  is not  true for the full range of parameters $p\leq p_1\leq \infty$ as stated there, since it relies on the fact that for $(\tilde{s},p_1)$-atoms on the unit sphere no moment conditions are needed where $\tilde{s}=s-\frac{1}{p_1}-(n-b)(1/p-1/p_1)>0$ with $1\leq b\leq n-2$ and $s>\frac{n-b}{p}$. We provide a counterexample here. Choosing 
\[n=10, \quad p=\frac 12=p_1, \quad b=n-2=8,\]
implies  $s>\frac{n-b}{p}=4$, thus taking $s=5$ we have $\tilde{s}=s-\frac 1p=s-2=3$ and 
\[
\sigma_{p_1}^{(n-1)}-\tilde{s}=(n-1)\left(\frac{1}{p_1}-1\right)-\tilde{s}=9(2-1)-3=6>0.
\] 
But this yields 
\[
M=\max([\sigma_{p_1}^{(n-1)}-\tilde{s}],-1)=6>0,
\]
meaning that moment conditions are sometimes needed indeed for $(\tilde{s},p_1)$-atoms in this situation. 
\end{rem}

\section{An application of the radial Strauss lemma}\label{sec:5}

In this section  we will show as an application of our results, that interpolation between Sobolev  inequalities for B- and F-spaces and the generalized Strauss inequalities obtained in Theorem~\ref{thm-strauss} yields weighted Sobolev inequalities of Caffarelli-Kohn-Nirenberg type for $G$-invariant functions. \\

We start with the relevant Sobolev embeddings  needed later on. Corresponding results on $\rn$ may be found in \cite{SiTr95}. \\

\begin{lem}\label{lem-sobolev-emb} Let $(M^n,g)$ be a Riemannian manifold of bounded geometry. Moreover, let  $0<p,q\leq \infty$ ($p<\infty$ for $F$-spaces) and $\sigma_p^{(n)}\define n(\frac{1}{p}-1)_+<s<\frac np$, where $n=\dim\, M$.
\begin{itemize}
\item[(i)] Let $p< r < \infty$ and $r$ such that 
\[
s-\frac np=-\frac nr.
\]
Then   
\begin{equation*}
F^s_{p,q}(M)\hookrightarrow L_{r}(M).
\end{equation*}
\item[(ii)] Let $q\leq r < \infty$ and $r$ such that 
\[
s-\frac np=-\frac nr.
\]
Then  
\begin{equation*}
B^s_{p,q}(M)\hookrightarrow L_{r}(M).
\end{equation*}
\end{itemize}
\end{lem}

\begin{proof} 
\begin{itemize}
\item[(i)] As for $F$-spaces, we know that corresponding  embeddings hold true on $\rn$, $F^s_{p,q}(\rn)\hookrightarrow L_{r}(\rn)$ with our restrictions on the parameters, cf. \cite[Rem.~3.3.5]{SiTr95}. This  together with the fact that our F- and $L_p$-spaces on $M$ can be  defined via localization yields the result. In particular, we have 
\[
\|f\big|L_p(M)\|=\left(\sum_{\alpha\in I}\|(h_{\alpha}f)\circ \kappa_{\alpha}\|^p_{L_p(\rn)}\right)^{1/p}. 
\]
\item[(ii)] Since $B$-spaces are defined via interpolation of $F$-spaces, see Definition~\ref{F-koord}, let $s_0,s_1\in \real$ with  $s_0<s_1$,  $\sigma_p^{(n)}<s_i<\frac np$,  and  
\[
s_i-\frac np=-\frac {n}{r_i}\quad \text{for} \quad i=0,1. 
\]
Then $\sigma_p^{(n)}<s=(1-\Theta)s_0+\Theta s_1<\frac np$ and  
we see that 
\begin{align*}
B^s_{p,q}(M)&=\left(F^{s_0}_{p,p}(M),F^{s_1}_{p,p}(M)\right)_{\Theta,q}\\
&\hookrightarrow  \left(F^{s_0}_{p,p}(M),F^{s_1}_{p,p}(M)\right)_{\Theta,r}\\
&\hookrightarrow \left(L_{r_0}(M),L_{r_1}(M)\right)_{\Theta,r}=L_{r}(M),
\end{align*}
where 
$\frac{1}{r}=\frac{1-\Theta}{r_0}+\frac{\Theta}{r_1}=\frac 1p-\frac sn$. The first embedding is a property of the real interpolation method, i.e., it holds $(\cdot, \cdot)_{\Theta,q}\hookrightarrow (\cdot, \cdot)_{\Theta,r}$ if $q\leq r$, cf. \cite[Sect. 1.3.3]{Tri78}. This completes the proof. 
\end{itemize}
\end{proof}

\begin{rem}
 Compared to \cite[Rem.~3.3.5]{SiTr95}, we adapted the assumptions on the parameters in Lemma \ref{lem-sobolev-emb} according to our needs for what follows next. Instead of assuming 
$\sigma_p^{(n)}\define n(\frac{1}{p}-1)_+<s<\frac np$ we could alternatively use $1\leq r<\infty$ and $s>0$, which together with the relation $s-\frac np=-\frac nr$ implies the same restrictions on $s$ as in Theorem \ref{lem-sobolev-emb} above. 
\end{rem}

Now we are in a position to formulate some weighted Sobolev inequalities of Caffarelli-Kohn-Nirenberg-type. 

\begin{thm}\label{thm:CKN-ineq}  Let $(M^n,g)$ be a Riemannian manifold and let $G$ be a compact subgroup of  $\mathrm{Isom}(M,g)$. We assume that the $G$-action on $M$ fulfills Assumptions~\eqref{est-Hjk} and~\eqref{est-Pjk}. Moreover, put  $\tilde{n}\define \dim (\tilde{M}/G)$ and let   $1\leq p, \tilde{q}\leq \infty$  and 
\[\max\left(\sigma_p^{(n)}, \frac{\tilde{n}}{p}\right)<s<\frac np, \qquad \text{where}\qquad  \sigma_p^{(n)}=n\left(\frac 1p -1 \right)_+.\]
Furthermore, let $ p<r<q\leq \infty$ and $r$ such that 
\[
s-\frac np =-\frac nr.
\]
Then  for functions $f$ with $\supp f\subset \tilde{M}$  we have 
\begin{equation*}
\|w^{\left(1-\frac rq\right)}f|L_q(M)\|\leq c \|f|A^{s}_{p,\tilde{q}}(M)\|,
\end{equation*}
where $A\in \{B,F\}$ and $w(x)=\vol_{n-\tilde{n}} (G\cdot x)^{\frac 1p}$ if $x\in \tilde{M}$.  For $A=B$ we additionally require $\tilde{q}\leq r$.
\end{thm}

\begin{proof}
We can rewrite the Strauss estimates from Theorem~\ref{thm-strauss} as 
\begin{equation}\label{strauss-1}
\| f|L_{\infty}(w,M)\|\leq c \|f|A^s_{p,\tilde{q}}(M)\|,\qquad A\in \{B,F\},
\end{equation}
where by $L_s(w,M)$, $0<s\leq \infty$,  we denote the weighted Lebesgue spaces  $L_s(w,M)\define \{ f\ |\ wf\in L_s(M)\}$ with nonnegative measurable weight function $w$. Now real interpolation $(\cdot,\cdot)_{\Theta,q}$ of  estimate \eqref{strauss-1}  with the Sobolev inequalities  
\[
\|f|L_r(M)\|\leq c \|f|A^s_{p,\tilde{q}}(M)\| 
\]
from Lemma \ref{lem-sobolev-emb} yields the spaces $A^s_{p,\tilde{q}}(M)$ on the right hand side and the weighted spaces $\left(L_r(M),L_{\infty}(w,M)\right)_{\Theta,q}$  on the left hand side, where $\frac 1q=\frac{1-\Theta}{r}+\frac{\Theta}{\infty}=\frac{1-\Theta}{r}$ and therefore $\Theta =1-\frac rq$. We use results on  real interpolation for weighted $L_p$-spaces from  \cite{Frei78} and  obtain
\[
\left(L_r(w_0\define \text{id},M),L_\infty(w_1\define w ,M)\right)_{\Theta,q}
=L_q(w_1^{\Theta},M), 
\]
where 
\[w_1^{\Theta}=w^{\left(1-\frac rq\right)}=\vol_{n-\tilde{n}} (G\cdot x)^{\frac 1p\left(1-\frac rq\right)}.\]  We point out that the constants involved in the corresponding norms are independent of the weights $w_0$ and $w_1$. Furthermore,  $w_1=\vol_{n-\tilde{n}} (G\cdot x)^{1/p}$  is a suitable weight  satisfying the desired properties from  \cite{Frei78} (measurable function with values in $[0,\infty]$ and $\mu(x\in M: \vol_{n-\tilde{n}} (G\cdot x)^{1/p}=0\cup \vol_{n-\tilde{n}} (G\cdot x)^{1/p}=\infty)=0$, the latter being true because  $G\cdot x$ is compact since $G$ is compact and for $x\in \tilde{M}$ we always have $\vol_{n-\tilde{n}} (G\cdot x)\neq 0$). This leads to the desired inequality of Caffarelli-Kohn-Nirenberg type,
\[
\|\vol_{n-\tilde{n}} (G\cdot x)^{\frac 1p\left(1-\frac rq\right)}f|L_q(M)\|\leq c \|f|A^s_{p,\tilde{q}}(M)\|.
\]
\end{proof}

\appendix
\section{On a source of examples}\label{app}

The aim of this appendix is to provide a class of examples of manifolds $(M,g)$ of bounded geometry, see Definition~\ref{defi_bdgeom}, that satisfy  the Assumptions~\eqref{est-Hjk}--\eqref{est-Pjk}. More precisely we will provide examples that fullfill:

\begin{assumption} \label{ass:G2} There is a $G$-adapted covering such that the following holds. 
 \begin{itemize}
  \item[(i)] For all $x\in B(x_{j,k,\ell}, 2^{-j+1})$ we have 
  \begin{align}\label{est-Hjk2} \mathcal{H}_{j,k} \gtrsim \text{vol}_{n-\tilde{n}}(G\cdot x) 2^{j(n-\tilde{n})}\end{align}
  and
  \begin{align}\label{est-Hjk2b} \mathcal{H}_{j,k} \sim \text{vol}_{n-\tilde{n}}(G\cdot x_{j,k,\ell}) 2^{j(n-\tilde{n})},\end{align}  
 where $\tilde{n}$ denotes the dimension of the orbit space. 
  \item[(ii)] There are constants $b>0$ and $\epsilon>0$  such that for all $x\in M$ with $\text{dist}_g(x, M\setminus \tilde{M})\geq r$ it holds 
  \begin{align} \label{est-volGx2} \text{vol}_{n-\tilde{n}}(G\cdot x) \gtrsim \min\{ r^b, \epsilon\}.  \end{align}
 \end{itemize}
where the constants do not depend on $j$, $k$, $\ell$, $x$ and $r$.
\end{assumption}

\begin{rem}\label{rem:ass}
\begin{enumerate}[(i)]
\item Note that \eqref{est-volGx2} and \eqref{est-Hjk2} are simply the same as \eqref{est-volGx} and  \eqref{est-Hjk}, respectively. In order to obtain \eqref{est-Pjk} we note that the considerations above Assumption~\ref{ass:G} imply that for an $x_{j,k_j}$ not just one $x_{j+1,k_{j+1}}$ with $\text{dist}_{M/G} (x_{j,k_j}, x_{j+1,k_{j+1}})\leq 2^{-j}$ exists. There is also another $x_{j+1,\tilde{k}_{j+1}}$ with the same property that additionally satisfies $\text{dist}_{M/G} (x_{j+1,\tilde{k}_{j+1}}, x_{j+1,k_{j+1}})\geq 2^{-j+1}$. This implies 
\begin{align*}\ \ \vol(P_{j,k_j}\setminus P_{j+1, k_{j+1}})&\geq \vol (P_{j+1, \tilde{k}_{j+1}})\geq \mathcal{H}_{j+1,\tilde{k}_{j+1}} \inf_{\ell} B(x_{j+1,\tilde{k}_{j+1},\ell}, 2^{-j-1}) \\ & \gtrsim 2^{-jn} \mathcal{H}_{j+1,\tilde{k}_{j+1}},\end{align*} where the last estimate follows since by the bounded geometry of $(M,g)$ the volume of balls is uniformly comparable to the volume of Euclidean balls. Now \eqref{est-Hjk2} and \eqref{est-Hjk2b} imply $\mathcal{H}_{j+1,\tilde{k}_{j+1}}\gtrsim \mathcal{H}_{j,k_j}$ and we obtain \eqref{est-Pjk}. Thus, in total we have seen that Assumption~\ref{ass:G2} implies  Assumption~\ref{ass:G}. 
\item In the special case that $M$ only consists of the principal stratum, \eqref{est-volGx2} is just a lower bound on the volume of the orbits (Here we used the convention that $\text{dist}_g(x, \varnothing)=0$). In this case, Assumption~\ref{ass:G2} is already satified if the manifold is a foliation of bounded geometry, cf. \cite{AKL14} and also compare Remark~\ref{rem:folbdd} below.
\end{enumerate}
\end{rem}

The special case above shows that if one works with manifolds that are  a foliation of bounded geometry in the sense of Definition~\ref{Gbbgeo} away from the singular strata, Assumption \ref{ass:G2} is satified away from the strata. Thus, our main concern will be  the geometry near the singular strata. Before providing our class of examples we shortly recall basic properties of the stratification induced by the group action.

\subsection{Manifolds with symmetries}

The action of $G$ on $M$ is denoted by the map $(g,x)\in G\times M\mapsto g\cdot x\in M$ where 
$G\cdot x= \{ g\cdot x\, |\, g\in G\}$ is the orbit through $x$ and $S_G(x)= \{ g\in G\, |\, g\cdot x=x\}$ is the stabilizer of $x$.

If $y=g\cdot x$, then  $S_G(y)=g\cdot S_G(x)\cdot g^{-1}$.
The \emph{orbit type} of $x\in M$ is defined to be the conjugacy class $(S_G(x))\define \{  g\cdot S_G(x)\cdot g^{-1}\ |\ g\in G\}$. Thus, $(S_G(x))=(S_G(g\cdot x))$ for all $x\in M$ and $g\in G$. Inclusion defines a partial ordering on the set of all conjugacy classes. The minimal element of this partial ordering is unique and called the \emph{principal orbit type} $(P)$.  For a conjugacy class $(H)$ we set
 \[ M_{(H)}\define \{x\in M\ |\ S_G(x)\in (H)\} \]
 and call this set the \emph{$(H)$-stratum of $M$}. In particular, we call $\tilde{M}\define M_{(P)}$ the \emph{principal stratum}. Furthermore, we let $\pi\colon M\to M/G$ be the canonical projection and put $\hat{M}_{(H)}\define \pi(M_{(H)})$.

In the next theorem we collect several well-known facts:

\begin{thm}\cite{Bre72} Let $(M,g)$ be a Riemannian manifold, and let $G$ be a compact subgroup of  $\mathrm{Isom}(M,g)$. Then
\begin{itemize}
 \item[(i)] The principal stratum $\tilde{M}$ is a dense and open subset of $M$ and $\tilde{M}/G$ is a smooth manifold such that $\pi|_{\tilde{M}}$ is a submersion.
 \item[(ii)] $M/G$ is Hausdorff and $\pi\colon M\to M/G$ is proper. 
 \item[(iii)] $M_{(H)}$ is a submanifold for every $(H)$.
\end{itemize}
\end{thm}

\subsection{Manifolds with group action and adapted bounded geometry}

Next we propose a definition for the notion of bounded geometry that is adapted to the group action. Up to this point  we simply collected assumptions that imply what was needed for our purposes. In view of further investigations, e.g. when looking for 'G-adapted' coordinates that will provide admissible trivializations in the sense of \cite[Def. 3.7]{GS13},  more assumptions on the geometry near the singular strata have to be expected.  But for the purpose of this paper the following is sufficient:

\begin{defi}\label{Gbbgeo}
  Let $(M,g)$ be a Riemannian manifold with $G$ being a compact subgroup of  $\mathrm{Isom}(M,g)$. 
 We say that $(M, g, G)$ has \emph{$G$-adapted bounded geometry} if  
 \begin{enumerate}[(i)]
\item $(M,g)$ has bounded geometry.
 \item There is an $r_0>0$ such that for the $G$-action on $M_{r_0}\define M\setminus U_{r_0}(M\setminus \tilde{M})$ where $U_{r_0}(M\setminus \tilde{M})=\{x\in M\ |\ \text{dist}(x, M\setminus \tilde{M})<r_0\}$ we have:
\begin{enumerate}[(a)]
\item The injectivity radii of the orbits in $M_{r_0}$ are uniformly bounded from below by a positive $\tilde{\epsilon}$.
\item  There is an $\epsilon>0$ such that there is a collar around each orbit in $M_{r_0}$ (a tubular neighbourhood), i.e., 
for all $x\in M_{r_0}$ and $g_1, g_2\in G$ with $g_1\neq g_2$ the normal balls $B_\epsilon^\perp(g_1\cdot x)$ and $B_\epsilon^\perp(g_2\cdot x)$ are disjoint, where
\begin{align*}B_\epsilon^\perp(x) \define \left\{ \right.&z\in M\ |\ \text{dist}_M (x, z) < \epsilon,\\ 
&\left.\exists \delta_0 >0\, \forall \delta<\delta_0: \text{dist}_M (x, z) = \text{dist}_{M} (B_\delta^{G\cdot x} (x), z)\right\}
\end{align*}
with $B_\delta^{G\cdot x} (x) = \{u \in G\cdot x\ |\ \text{dist}_{G\cdot x} (u, x) < \delta \}$ and $\text{dist}_M$ and $\text{dist}_{G\cdot x}$ denote the induced distance functions on $M$ and $G\cdot x$, respectively.
\item The O'Neill tensors (see Remark \ref{rem:folbdd} below) of the foliation $M_{r_0}\to M_{r_0}/G$ are totally bounded, i.e., $\Vert \nabla^kT\Vert_{L^\infty(M_{r_0})}<\infty$ and $\Vert \nabla^kA\Vert_{L^\infty(M_{r_0})}<\infty$.
\end{enumerate}
\item For each orbit type $(H)$ there is a constant $C_{(H)}$ such that for all $x\in M_{(H)}$ and every geodesic $\gamma\colon (0,2r_0)\to \tilde{M}/G$  with $\lim_{t\searrow 0}\gamma(t)=x$ and $(\lim_{t\searrow 0}\dot\gamma(t))\perp M_{(H)}$ such that $\inf_{r\in (0,2r_0)} r^{-b_{(H)}}\vol_{n-\tilde{n}} (G\cdot \gamma(r))\geq C_{(H)}$.
\item For all $r\in (0,r_0)$, $r'<r/2$ it holds 
\[ U_ {r'}(G\cdot x)\sim \vol_{n-\tilde{n}} (G\cdot x) (r')^{\tilde{n}}\]
for all $x$ with $0<\text{dist} (x, M\setminus \tilde{M})=r<r_0$.
\end{enumerate}
\end{defi}

 \begin{rem}\label{rem:folbdd} Let $M=\tilde{M}$, i.e., we only have the principal stratum. Then the above definition is simply the definition of a Riemannian foliation with bounded geometry, see \cite[Def. 8.1]{AKL14}, where   the leaves are the orbits.  Moreover, from \cite[Prop. 8.6]{AKL14} we deduce that 
  in this case (ii) implies (i).
  
  The orbits (or in general the leaves) induce an orthogonal decomposition of each tangent space $T_xM$ in a vertical subspace, tangent to the orbit $G\cdot x$, and the horizontal subspace. The projection of a vector field $X$ on $M$ onto its vertical resp. horizontal part is denoted by $V(X)$ resp. $H(X)$. The O'Neill tensors, cf. \cite{ON66},  $T$ and $A$ are $(1,2)$ tensors with
  \begin{align*} T_XY&=H(\nabla_{V(X)} V(Y))+V(\nabla_{V(X)} H(Y))\quad \text{and}\\
   A_XY&=H(\nabla_{H(X)} V(Y))+V(\nabla_{H(X)} H(Y)).
   \end{align*}
 \end{rem}

\begin{rem}[$G$-adapted discretization]\label{rem:gd}
Let $(M,g)$ be a Riemannian manifold with $G$-adapted covering. We choose for any $j\in \mN$ a $G$-adapted $(2^{-j},2)$ discretization  as in Subsection~\ref{ssec:op} under the additional condition that the starting $(2^{-j},1)$-discretization $\{G\cdot x_{j,k,\ell}\}$ of $\tilde{M}$ is such that $\text{dist} (x_{j,k,\ell}, M\setminus \tilde{M})\geq 2^{-j-1}$. By triangle inequality this is always possible. We need this condition on the distance to the singular set since otherwise \eqref{est-Hjk2b} cannot be achieved ($\mathcal{H}_{j,k}\geq 1$ while $\vol_{n-\tilde{n}} (G\cdot x_{j,k, \ell})$ could converge to zero for $j\to \infty$).
\end{rem}

\begin{lem}
 Let $(M,g)$ be a Riemannian manifold with $G$-adapted bounded geometry. Let $\{x_{j,k,\ell}\}$ be a $G$-adapted covering as in the last remark. Then Assumptions~\ref{ass:G2} and \ref{ass:G} are fulfilled.
\end{lem}

\begin{proof}W.l.o.g. we assume that $\min\{ r_0, \epsilon, \tilde{\epsilon}\}=1$ holds for the constants in Definition~\ref{Gbbgeo} (this is just so  we can stick to $(2^{-j},\alpha=1)$-discretizations in Subsection~\ref{ssec:op} instead of adjusting the $\alpha$).  Condition (ii)(a) of Definition~\ref{Gbbgeo} implies that ${\inf_{x\in M_{r_0}}} \vol (G\cdot x)$ is positive and, hence, gives \eqref{est-volGx2} for all $x\in M_{r_0}$.
Condition (iii) of Definition~\ref{Gbbgeo} implies \eqref{est-volGx2} on $M\setminus M_{r_0}$ (and together with the above on all of $M$), since there are only finitely many orbit types and we can choose $b\define {\max}_{(H)} b_{(H)}$. 

Condition (ii)(b) in Definition~\ref{Gbbgeo} gives positive constants $C_i$ with 
\[C_1 \leq \vol (U_{2^{-j}}(G\cdot x)) (2^{-j\tilde{n}}\vol_{n-\tilde{n}}(G\cdot x))^{-1} \leq C_2.\]
Since the balls $B(x_{j,k,\ell}, 2^{-j+1})$ cover $U_{2^{-j}}(G\cdot x_{j,k,1})$ and the balls $B(x_{j,k,{\ell}}, 2^{-j-1})\subset U_{2^{-j}}(G\cdot x_{j,k,1})$ are disjoint, the bounded geometry of $(M,g)$ implies $2^{-jn} \mathcal{H}_{j,k}\sim \vol (U_{2^{-j}}(G\cdot x_{{j,k,\ell}}))$. Thus, we have \eqref{est-Hjk2b} for all $x_{j,k,\ell} \in M_{r_0}$ and by triangle inequality \eqref{est-Hjk2} for all $x\in M_{r_0}$. Condition (iv) from Definition ~\ref{Gbbgeo} gives 
together with the choice of $x_{j,k,\ell}$ as in Remark~\ref{rem:gd} that $\mathrm{vol}(U_{2^{-j-2}}(G\cdot x_{j,k,\ell}))\sim \text{vol}_{n-\tilde{n}} (G\cdot x_{j,k,\ell}) 2^{-j\tilde{n}}$. Then an analog argument as above implies the rest of Assumption~\ref{ass:G2} and by Remark~\ref{rem:ass} also Assumption~\ref{ass:G}.
\end{proof}

\begin{ex}
Consider  $SO(n)$ on a warped product $(\mS^{n-1}\times \mR_{\geq 0}, f^2(r)\sigma_{n-1}+ \d r^2)$. 
  Then for $x$ with $r=|x|>0$ we have $\vol_{n-\tilde{n}} (G\cdot x)= f(r)^{n-1}\vol (\mS^{n-1})$. In the particular case of the Euclidean space we have  $f(r)=r$ and  recover the Strauss lemma. For the hyperbolic space we have  $f(r)=\sinh r$ and  obtain an exponential decay.
\end{ex}

\bibliographystyle{alpha}

 \def\cprime{$'$}

\vfill

{\small
\begin{minipage}[t]{0.45\textwidth}
\noindent
Nadine Gro\ss{}e\\
Mathematical Institute\\  
 University of  Freiburg \\
 Eckerstr. 1\\ 
 79104 Freiburg \\
 Germany\\[1ex]
\texttt{nadine.grosse@math.uni-freiburg.de}
\end{minipage}\hfill 
\begin{minipage}[t]{0.45\textwidth}
\noindent
Cornelia Schneider\\
Applied Mathematics III\\
University of Erlangen--Nuremberg\\
Cauerstra\ss{}e 11\\
91058 Erlangen\\
Germany\\[1ex]
\texttt{schneider@math.fau.de}
\end{minipage}

}

\end{document}